\pdfoutput=1
\documentclass[11pt]{article}

\usepackage[utf8]{inputenc}
\def\final{1}
\usepackage{color}
\usepackage[small,raggedright]{subfigure}
\usepackage{url}
\usepackage{graphicx}
\usepackage{amsmath, amscd}
\usepackage{amsthm}
\usepackage{amscd}
\usepackage{bm}
\usepackage[plainpages=false]{hyperref}
\usepackage{pinlabel}

\usepackage{titlesec}
\titleformat{\section}
	{\normalfont\large\bfseries\filcenter}{\thesection.}{1 ex}{}
\titleformat{\subsection}
	{\normalfont\normalsize\bfseries\filcenter \raggedright}{\thesubsection.}{1 ex}{}

\usepackage[charter]{mathdesign}
\usepackage[mathcal]{eucal}

\usepackage[margin=1.2 in]{geometry}

\usepackage[square,numbers,sort&compress]{natbib}

\ifnum\final=0
\newcommand{\mynote}[1]{\marginpar{\tiny\sf #1}}
\else
\newcommand{\mynote}[1]{}
\fi

\usepackage{theomac}
\usepackage{thmtools,thm-restate}

\declaretheorem[within=section]{theorem}
\declaretheorem[sibling=theorem,name=Corollary]{corollary}
\declaretheorem[sibling=theorem]{lemma}
\declaretheorem[sibling=theorem, name=Proposition]{prop}
\declaretheorem[sibling=theorem, name=Remark]{remarknum}

\newtheorem*{conjecture}{Conjecture}
\newtheorem*{remark}{Remark}

\renewcommand{\eqref}[1]{(\ref{eq:#1})}

\newcommand{\JMat}[1]{ \left( \begin{array}{rr} #1 \end{array} \right)}
\newcommand{\Id}{\operatorname{Id}}

\newcommand{\eop}{\hfill$\qed$}
\newcommand{\mc}{\mathcal}

\newcommand{\R}{\mathbb{R}}
\newcommand{\Q}{\mathbb{Q}}
\newcommand{\Z}{\mathbb{Z}}
\newcommand{\C}{\mathbb{C}}
\newcommand{\N}{\mathbb{N}}
\newcommand{\bbH}{\mathbb{H}}

\DeclareMathOperator{\Dih}{Dih}
\DeclareMathOperator{\Dic}{Dic}
\DeclareMathOperator{\Sym}{Sym}
\DeclareMathOperator{\Map}{Map}
\DeclareMathOperator{\Mod}{Mod}
\DeclareMathOperator{\HH}{H}
\DeclareMathOperator{\Mat}{Mat}

\DeclareMathOperator{\Out}{Out}
\DeclareMathOperator{\Aut}{Aut}
\DeclareMathOperator{\End}{End}
\DeclareMathOperator{\Hom}{Hom}
\DeclareMathOperator{\tr}{tr}
\DeclareMathOperator{\Tr}{Tr}
\DeclareMathOperator{\trd}{trd}

\DeclareMathOperator{\Sp}{Sp}
\DeclareMathOperator{\SO}{SO}
\DeclareMathOperator{\SL}{SL}
\DeclareMathOperator{\GL}{GL}
\DeclareMathOperator{\OO}{O}

\DeclareMathOperator{\Homeo}{Homeo}
\DeclareMathOperator{\nrd}{nrd}

\DeclareMathOperator{\inn}{inn}
\DeclareMathOperator{\Epi}{Epi}
\DeclareMathOperator{\Cyc}{Cyc}
\DeclareMathOperator{\Alt}{Alt}

\DeclareMathOperator{\SU}{SU}

\DeclareMathOperator{\Diff}{Diff}

\usepackage{authblk}

\author[1]{Fritz Grunewald\footnote{Paper written posthumously}}
\author[2]{Michael Larsen\thanks{mjlarsen@indiana.edu}}
\author[3]{Alexander Lubotzky\thanks{alex.lubotzky@mail.huji.ac.il}}
\author[4]{Justin Malestein\thanks{justinm@math.uni-bonn.de}}
\affil[1]{ Mathematisches Institut der Heinrich-Heine-Universit\"at D\"usseldorf, Universit\"atsstr. 1, 40225 D\"usseldorf, Germany}
\affil[2]{Department of Mathematics, Indiana University, Bloomington, IN, 47405, USA}
\affil[3]{Einstein Institute of Mathematics, The Hebrew University of Jerusalem, Jerusalem, 91904, Israel}
\affil[4]{Mathematisches Institut, Universit\"at Bonn, Endenicher Allee 60, 53115 Bonn, Germany}
\title{Arithmetic quotients of the mapping class group}
\date{}

\begin{document}
\maketitle
\begin{abstract}
\begin{normalsize}
%!TEX TS-program = ./Scripts/wrapper
To every $\Q$-irreducible representation $r$ of a finite group $H$, there 
corresponds a simple factor $A$ of $\mathbb{Q}[H]$ with an involution $\tau$. To this pair 
$(A,\tau)$, we associate an arithmetic group $\Omega$ consisting of all $(2g-2)\times (2g-2)$ 
matrices over a natural order $\mathfrak{O}^{op}$ of $A^{op}$ which preserve a natural skew-Hermitian 
sesquilinear form on $A^{2g-2}$.
We show that if $H$ is
generated by less than $g$ elements, then $\Omega$ is a virtual quotient of the mapping class group
$\Mod(\Sigma_g)$, i.e. a finite index subgroup of $\Omega$ is a quotient of a finite index subgroup
of $\Mod(\Sigma_g)$. This shows that $\Mod(\Sigma_g)$ has a rich family of arithmetic quotients
(and ``Torelli subgroups'') for which the classical quotient $\Sp(2g, \Z)$ is just a first case in a list, 
the case corresponding to the trivial group $H$ and the trivial representation. Other pairs of $H$
and $r$ give rise to many new arithmetic quotients of $\Mod(\Sigma_g)$ which are defined over
various (subfields of) cyclotomic fields and are of type $\Sp(2m), \SO(2m,2m),$
and $\SU(m,m)$ for arbitrarily large $m$.

\end{normalsize}
\end{abstract}

\section{Introduction}
\label{section:intro}
%!TEX TS-program = ./Scripts/wrapper
Let $\Sigma = \Sigma_g$ be a closed surface of genus $g \geq 2$ and $\Mod(\Sigma)$ the mapping class group of 
$\Sigma$, which is the group of orientation-preserving homeomorphisms of $\Sigma$ modulo those isotopic
to the identity. It is well-known that there is an epimorphism from $\Mod(\Sigma)$ onto the arithmetic
group $\Sp(2g, \Z)$.
The purpose of this paper is to show that many more arithmetic groups are quotients of $\Mod(\Sigma)$.
Specifically, for every pair $(H, r)$, where $H$ is a finite group minimally generated by $d(H) < g$
elements and $r$ is a nontrivial irreducible $\Q$-representation of $H$, we associate a finite index subgroup
$\Gamma_H$ of $\Mod(\Sigma)$, an arithmetic group $\Omega_{H, r}$ (whose exact structure
will be described below) and a homomorphism $\rho_{H, r} : \Gamma_H \to \Omega_{H, r}$
whose image is of finite index in $\Omega_{H, r}$. We call such a homomorphism a
{\it virtual epimorphism} of $\Mod(\Sigma)$ onto $\Omega_{H,r}$. (Of course, via the
induced representation, every virtual homomorphism also gives a representation of the
full group $\Mod(\Sigma)$, but it is more natural to study the original representation.)
The homomorphism
$\Mod(\Sigma) \to \Sp(2g, \Z)$ is the special case where $H$ is the trivial group. The
homomorphism from a finite index subgroup of $\Mod(\Sigma_g)$ to $\Sp( 2(g-1), \Z)$
studied in \cite{LuboMeir} and \cite{MaleSout} is the one corresponding to the case $H = \Z/2\Z$
and $r$ the nontrivial one-dimensional representation. The first to show that $\Mod(\Sigma)$ has
a large collection of virtual arithmetic quotients beside $\Sp(2g, \Z)$ was Looijenga \cite{Looi}
who analyzed the case where $H$ is abelian. In  \cite{Looi}, \cite{LuboMeir}, \cite{MaleSout}, all representations
have degree bounded by $2g$ and the target arithmetic groups are limited, while we obtain irreducible representations
of arbitrarily large degree and a wider range of target arithmetic groups.

The homomorphism $\rho$ is obtained from the action of $\Mod(\Sigma)$ (or more precisely, of a finite index
subgroup of $\Mod(\Sigma)$) on the first integral homology group of a finite Galois cover of $\Sigma$
corresponding to $H$;
equivalently, on $\overline{R} \overset{\text{def}}{=} R/[R, R]$ where $R$ is a finite index normal subgroup 
of $T_g = \pi_1(\Sigma_g)$ with $T_g/R \cong H$. This action on the homology of finite covers was also studied
by Koberda, but with the goal of showing that one can determine whether a mapping class is periodic,
reducible or pseudo-Anosov by considering enough finite covers \cite{Kobe}.

Our work is also related to that of Grunewald--Lubotzky \cite{GrunLubo} whose results we will use in our proofs.
Recall that by the Dehn-Nielsen-Baer Theorem, $\Mod(\Sigma_g)$
is naturally isomorphic to an index $2$ subgroup of $\Out(T_g)$, the outer
automorphism group of $T_g = \pi_1(\Sigma_g)$. (Allowing orientation-reversing homeomorphisms would give all of $\Out(T_g)$.) In
\cite{GrunLubo}, Grunewald and Lubotzky studied the analogous problem where $\Out(T_g)$ is replaced by $\Aut(F_n)$ the
automorphism group of $F_n$, the free group of rank $n$. While there are a few similarities between the papers, 
the theory developed in the current paper requires 
several advances (which we will discuss), and the main techniques of proof in \cite{GrunLubo}, as far 
as we can determine, are unusable in this case. The difference is nicely illustrated already in the
easiest case; while the fact that $\pi: \Aut(F_n) \to \GL(n, \Z)$ is surjective is easy to establish,
the fact that $\Mod(\Sigma_g) \to \Sp(2g, \Z)$ is a classical but nontrivial result due to
Burkhardt and Clebsch-Gordan \cite{Burk}

We begin by describing the general procedure for obtaining representations to arithmetic groups
from $\Mod(\Sigma_g, \ast)$, the mapping class group
fixing the point $\ast \in \Sigma_g$ up to isotopy fixing $\ast$. We will denote these representations
also by $\rho$. 
A short argument (Section \ref{section:forgetmarkedpt}) shows that the representations descend (virtually) to $\Mod(\Sigma_g)$.
Via the natural map $\Mod(\Sigma_g, \ast) \to \Aut(T_g)$, we can identify $\Mod(\Sigma, \ast)$ with its image,
which is of index $2$ and which we denote by $\Aut(T_g)^+$.  
Let $p: T_g \to H$ be a surjective homomorphism with kernel $R$ onto a finite group $H$, and let 
$\Gamma_{H,p} = \{\gamma \in \Aut(T_g)^+ \;\; | \;\; p \circ \gamma = p \}$. Then,
$\Gamma_{H,p}$ is a finite index subgroup of $\Aut(T_g)^+$ which preserves $R$ and acts on $\overline{R}$
as a $\Z[H]$-module. Thus, $\Gamma_{H,p}$ acts by $\Q[H]$-automorphisms on $\hat{R} = \overline{R} \otimes_{\Z} \Q$
and we obtain a representation $\rho_{H,p}: \Gamma_{H, p} \to \Aut_{\Q[H]}(\hat{R})$.
(See Section \ref{section:liftrep} for a discussion of these and some other properties of $\rho_{H,p}$.)

By an analogue of a classical result of Gasch\"utz, the module $\hat{R}$ can be identified precisely.
Note that here, and elsewhere unless stated otherwise, $\Q$ is considered to be the trivial $\Q[H]$-module.
\begin{restatable}{prop}{surfgasc}
\label{prop:surfacegaschutz}
Let $T = T_g$ be the fundamental group of a surface $\Sigma$ of genus $g \geq 2$. Let $R$ be a finite index
normal subgroup of $T$ and $H = T/R$. Then $\hat{R} = \Q \otimes_\Z (R/[R,R])$ is isomorphic as a 
$\Q[H]$-module to $\Q^2 \oplus \Q[H]^{2g-2}$.
\end{restatable}
This result was proved in \cite{ChevWeil}. We give an alternate proof in Section \ref{section:gaschutz}.
The group ring $\Q[H]$ decomposes
as a direct sum of simple algebras $\Q[H] = \Q \oplus \bigoplus_{i=1}^\ell A_i$  where each $A_i$
is the ring of $m_i \times m_i$ matrices over a division algebra $D_i$ with a number field
$L_i/\Q$ as its center. Consequently, $\hat{R} \cong \Q^{2g} \oplus \bigoplus_{i=1}^\ell A_i^{2g-2}$,
and so we can project to representations $\rho_{H,p,i}: \Gamma_{H,p} \to \Aut_{A_i}(A_i^{2g-2})$.
The action on $\Q^{2g}$ is the standard symplectic representation.

Up to now, the above procedure follows that of \cite{GrunLubo} with the important exception that 
Grunewald and Lubotzky use the actual theorem of Gasch\"utz in place of Proposition \ref{prop:surfacegaschutz}.
Via such representations,  they show that $\Aut(F_n)$ virtually surjects
onto a rich class of arithmetic groups, including $\SL_{\ell(n-1)}(\Z)$ and $\SL_{\ell(n-1)}(\mc O_m)$
where $\ell$ ranges over all positive integers, $m$ ranges over all integers $\geq 3$, and $\mc O_m$ is the ring of integers in the number
field generated by $\Q$ and a primitive $m$th root of unity.

What makes the surface case much more difficult (and interesting) is the fact that 
in our case $\hat{R}$ is equipped with a $\Q[H]$-valued skew-Hermitian sesquilinear form $\langle-,-\rangle$
on $\hat{R}$ with respect to the standard involution $\tau$ on $\Q[H]$ which is defined by $\tau(h) = h^{-1}$. This form can be defined in terms
of the group action $H$ and the natural symplectic structure on $\hat{R}$ coming from its identification with the rational first homology 
of the covering surface.
(See Section \ref{section:symplecticgaschutz} for a detailed discussion of this form.) 
As we will see in Section \ref{section:liftrep}, $\rho_{H,p}(\Gamma_{H,p})$ preserves this form.
The sesquilinear form descends to a nondegenerate $A_i$-valued form on each factor $M_i = A_i^{2g-2}$,
and the image of $\rho_{H,p, i}$ lies in $\Aut_{A_i}(A_i^{2g-2}, \langle-,- \rangle)$. 

On account of this, we
obtain an even richer class of arithmetic quotients. We illustrate it here by some simple to state examples
obtained by appropriate choices of $H$ and $p$. See Section \ref{section:finitegroups} for more. 
We denote by $\zeta_n$ the primitive $n$th root of unity, and by $\Q(\zeta_n)^+$
we denote the index $2$ subfield of $\Q(\zeta_n)$ which is fixed by the order $2$ Galois
automorphism of $\Q(\zeta_n)$ mapping $\zeta_n$ to $\zeta_n^{-1}$. By $\SU(m, m, \mc O)$, we denote
the subgroup of $\SL_{2m}(\mc O)$ preserving the Hermitian form of signature $(m,m)$, namely
$$ \langle (a_1, \dots, a_m, b_1, \dots, b_m) , (a_1', \dots, a_m', b_1', \dots, b_m') \rangle
 = \sum_{i=1}^m (a_i \overline{a_i'} - b_i \overline{b_i'})$$
where $\mc O$ is the ring of integers in $\Q(\zeta_n)$ and $\overline{\phantom{a}}$ is
the order $2$ automorphism of $\Q(\zeta_n)$ just described.

\begin{theorem} \label{theorem:examplesofquotients}
For a fixed $g \geq 2$, there are virtual epimorphisms of $\Mod(\Sigma_g)$ onto the following arithmetic groups:
\begin{itemize}
\item[(a)] $\Sp(2m(g-1), \Z)$ for all $m \in \N$,
\item[(b)] for all $m \in \N$ and $n \geq 3$, the group $\Sp(4m(g-1), \mc O)$ where $\mc O$ is the ring of integers in $\Q(\zeta_n)^+$,
\item[(c)] for all $m \in \N$ and $n \geq 3$, the group $\SU(m(g-1), m(g-1), \mc O)$
where $\mc O$ is the ring of integers in $\Q(\zeta_n)$,
\item[(d)] for all $m \in \N$ and $n \geq 3$, an arithmetic group of type $\SO(2m(g-1), 2m(g-1))$
(whose precise description will be given in Section \ref{section:finitegroups}).
\end{itemize}
\end{theorem}

\begin{theorem} \label{theorem:examplesofquotients2}
Let $L$ be an arbitrary subfield of a finite cyclotomic extension of $\Q$ and $\mc O$ the ring of integers of $L$. 
Then there is some $s_L$ and $N_L$ such that
for every $g \geq N_L$ and every $m \in \N$, there is a virtual epimorphism of $\Mod(\Sigma_g)$ onto
\begin{itemize}
\item $\Sp(2 m s_L (g-1), \mc O)$ if $L$ is a totally real field
\item $\SU(m s_L (g-1), m s_L (g-1), \mc O)$ if $L$ is a totally imaginary field.
\end{itemize}
\end{theorem}

Note that any subfield of a cyclotomic field is necessarily either totally real or totally imaginary. (See Lemma \ref{lemma:cyccenter}.)

Taking $g=2$ and $m = 1$ in (a) of Theorem \ref{theorem:examplesofquotients}, we obtain the following corollary which
has been proven previously by Korkmaz by a different method \cite{Kork} and by McCarthy via similar techniques \cite{McCa}.

\begin{corollary} \label{corollary:mods2large}
There is a virtual epimorphism from $\Mod(\Sigma_2)$ onto $\Sp(2, \Z) = \SL(2, \Z)$. In particular,
there is a virtual epimorphism of $\Mod(\Sigma_2)$ onto a free group, and hence $\Mod(\Sigma_2)$ is large.
\end{corollary}

We also deduce the main result of \cite{MasbReid} in Section \ref{section:ontoallfingrp}.

\begin{corollary} \label{corollary:surjectfinitegroups}
For every genus $g \geq 2$ and every finite group $G$, there is a finite index subgroup of $\Mod(\Sigma_g)$ 
which surjects onto $G$.
\end{corollary}

We now explain the main technical result of the paper. The standard involution $\tau$ of $\Q[H]$, defined above,
descends to an involution
on each simple factor $A_i$ of $\Q[H]$ (see Lemma \ref{lemma:algdecomp}). Let $K_i = L_i^\tau$
be the subfield of the center $L_i$ fixed by $\tau$, and let $\frak O_i$ be the image of $\Z[H]$ in $A_i$ which
is an order in $A_i$. Let $\mc G_{H, i}$ be the $K_i$-defined
algebraic group $\Aut_{A_i}(A_i^{2g-2}, \langle-,-\rangle)$, and let $\mc G_{H, i}^1$ be those elements of
reduced norm $1$ over $L_i$ (see Section \ref{subsection:gofinindex} for the definition of reduced norm).
Let $\mc G_{H, i}^1(\frak O_i)$ be the arithmetic subgroup
$\mc G_{H, i}^1 \cap \Aut_{\frak O_i}(\frak O_i^{2g-2})$.
Our main result says that, under a suitable condition, the image of $\Gamma_{H,p}$ contains a finite index
subgroup of $\mc G_{H, i}^1(\frak O_i)$. Namely, the suitable condition is that $p$ be {\em $\phi$-redundant},
which means that $p$ factors through a surjective map $\phi: T_g \to F_g$ where $F_g$ is the rank $g$
free group and the induced
map $p' : F_g \to H$ is {\em redundant}, i.e. $p'$ contains a free generator in its kernel.

\begin{theorem}\label{theorem:mainresult} Suppose $g \geq 3$ and $p: T_g \to H$ is $\phi$-redundant.
Then, for $\rho_{H, p, i},\; \Gamma_{H,p},\; \mc G_{H,i}^1(\frak O_i)$ as defined above,
$\rho_{H, p, i}(\Gamma_{H,p})$ is commensurable with $\mc G_{H,i}^1(\frak O_i)$.
\end{theorem}

It is the structure of $(A_i^{2g-2}, \langle-,-\rangle)$ which ultimately determines $\mc G_{H,i}^1(\frak O_i)$,
and this structure is intimately related to the representation theory of the finite group $H$.
The structure is determined by the pair $(A_i, \tau|_{A_i})$, and
the simple factors $A_i$ of $\Q[H]$ are in natural one-to-one correspondence with the nontrivial
irreducible $\Q$-representations $r_i$ of $H$ (and the factor $\Q$ of $\Q[H]$ corresponds to the
trivial representation). We can also extract information about the structure (Theorem \ref{theorem:Gtype} below)
by extending scalars to get $\R$-algebras and $\C$-representations as follows.
Let $n_i^2 = \dim_{L_i}(A_i)$. As we will see later
(Proposition \ref{prop:typefromrealrep}), $K_i$ is a totally real field. As $A_i$ is simple,
so are $A_i \otimes_{L_i} \C$ and $A_i \otimes_{K_i} \R$; consequently, $A_i \otimes_{L_i} \C \cong \Mat_{n_i}(\C)$,
and $A_i \otimes_{K_i} \R$ is isomorphic to one of $\Mat_{n_i}(\R), \Mat_{n_i}(\C),$ or  $\Mat_{n_i/2}(\bbH)$
where $\bbH$ denotes the Hamiltonian quaternions. Note that the projection $\Q[H] \to A_i$ induces a 
representation $H \to (A_i \otimes_{L_i} \C)^\times$ which by the isomorphism $A_i \otimes_{L_i} \C \cong \Mat_{n_i}(\C)$
gives an irreducible $\C$-representation $r_{i,\C}: H \to \GL_{n_i}(\C)$.

After extending scalars to $\C$, the group $\mc G_{H, i}$ ``becomes'' one of the classical
algebraic groups over $\C$. The following theorem describes the structure after extending scalars
and summarizes, essentially, how we obtain
the different types of arithmetic groups ($\Sp, \OO, U$) in Theorem \ref{theorem:examplesofquotients}.
(For definitions of first/second kind and symplectic/orthogonal type, see Section \ref{subsection:AutM}.)

\begin{theorem} \label{theorem:Gtype}
The group $\mc G_{H, i}$ is the group of $K_i$-points of a $K_i$-defined complex algebraic group $G$
with the additional property that
\begin{itemize}
\item $G \cong \Sp_{(2g-2)n_i}(\C)$ $\Leftrightarrow $ $r_{i, \C}(H)$ preserves a nondegenerate symmetric bilinear form $\Leftrightarrow$
$A_i \otimes_{K_i} \R \cong \Mat_{n_i}(\R)$  $\Leftrightarrow$ $\tau|_{A_i}$ is of first kind and orthogonal type.
\item $G \cong \OO_{(2g-2)n_i}(\C)$ $\Leftrightarrow$ $r_{i, \C}(H)$ preserves a nondegenerate alternating bilinear form $\Leftrightarrow$
$A_i \otimes_{K_i} \R \cong \Mat_{n_i/2}(\bbH)$ $\Leftrightarrow$ $\tau|_{A_i}$ is of first kind and symplectic type.
\item $G \cong \GL_{(2g-2)n_i}(\C)$ $\Leftrightarrow$ $r_{i, \C}(H)$ preserves no nonzero bilinear form $\Leftrightarrow$
$A_i \otimes_{K_i} \R \cong \Mat_{n_i}(\C)$ $\Leftrightarrow$ $\tau|_{A_i}$ is of second kind.
\end{itemize}
\end{theorem}

Theorem \ref{theorem:mainresult} and its proof show
that we have the following ``procedure'' to obtain a rich collection of arithmetic quotients
of $\Mod(\Sigma_g)$. We present this procedure in the form of a theorem.
Note that the homomorphism $\rho_{H, A}$ below is explicit and constructive.

\begin{theorem} \label{theorem:maintheoremprocedural}
Let $H$ be a finite group with $d(H) < g$ generators and let $A$ be 
a nontrivial simple component of $\Q[H]$. Let $\tau$ be the involution on $A$
induced by the standard one of $\Q[H]$, let $L$ be the center of $A$,
and let $K = L^\tau$ be the $\tau$-fixed subfield.  Set $M = A^{2g-2}$ with free basis $x_1, \dots, x_{g-1}, y_1, \dots, y_{g-1}$,
and endow $M$ with the skew-Hermitian sesquilinear form satisfying $\langle x_i, y_j \rangle = \delta_{ij}, 
\langle y_i, x_j \rangle = -\delta_{ij}, \langle x_i, x_j \rangle = 0,$ and $\langle y_i, y_j \rangle = 0$.
Set $\mc G$ to be the $K$-algebraic group of the $A$-automorphisms of $(M, \langle-,-\rangle)$
and $\mc G^1$ to be its elements of reduced norm $1$ over $L$. Set
$\Omega_{H,A} = \mc G^1(\frak O) = \mc G^1 \cap \Aut_{\frak O}(\frak O^{2g-2})$ 
where $\frak O$ is the order in $A$ which is the image of $\Z[H]$.
Then, there is a finite index subgroup $\Gamma_H < \Mod(\Sigma_g)$ and a homomorphism
$\rho_{H,A}: \Gamma_H \to \Omega_{H,A}$
whose image is of finite index.
\end{theorem}

Note that since there is a bijection between nontrivial simple components $A$ of $\Q[H]$
and nontrivial irreducible $\Q$-representations $r$ of $H$, we could, with appropriate rewording,
replace $A$ with $r$ in the above theorem. Thus, if $A$ corresponds to $r$, then
$\rho_{H,r} = \rho_{H, A}$ and $\Omega_{H, r} = \Omega_{H, A}$ are respectively
the representation and arithmetic group promised at the beginning. Note, in addition, that
in Theorem \ref{theorem:maintheoremprocedural}, we did not reference
the $\phi$-redundant homomorphism $p: T_g \to H$. The homomorphism $\rho$ does depend
on the choice of $p$, but under certain conditions (such as $g \gg d(H)$) we know that different choices
of $p$ lead to representations $\rho$ which are equivalent in a natural way. 
In Section \ref{subsection:rhouniqueness}, we discuss the notion of equivalence of representations $\rho$
and when they are known to be equivalent. 

Theorem \ref{theorem:examplesofquotients} is deduced from Theorem \ref{theorem:maintheoremprocedural}
by making special choices of finite groups $H$ and simple components $A$ of $\Q[H]$ and then
analyzing the resulting arithmetic groups, in some cases using Theorem \ref{theorem:Gtype}.
(See Section \ref{section:finitegroups}
for the corresponding irreducible representations $r$.)
For (a), we use $H = \Sym(m+1)$ and $A = \Mat_m(\Q)$. For (b) and $m \geq 3$, we use $H = \Alt(m+1) \times \Dih(2n)$ where
$\Dih(2n)$ is the dihedral group of order $2n$, and $A = \Mat_n(L)$ where $L$ is the field $\Q(\zeta_n)^+$.
For (c) and $m \geq 2$, we use $H = \Sym(m+1) \times \Z/n\Z$ and $A = \Mat_n(\Q(\zeta_n))$. For (d) and $m \geq 3$, we use
$H = \Alt(m+1) \times \Dic(4n)$ where $\Dic(4n)$ is the dicyclic group of order $4n$ (definition 
in Section \ref{section:finitegroups}) and $A = \Mat_n(D)$ where $D$ is the quaternion algebra
$$D = \left\{ \left( \begin{array}{rr} \alpha & \beta \\ - \overline{\beta} & \overline{\alpha} \end{array} \right) \in \Mat_2(\Q(\zeta_n))
 	\; | \; \alpha, \beta \in \Q(\zeta_n) \right\}.$$
For smaller $m$ in cases (b),(c),(d), see Section \ref{section:finitegroups} for the corresponding $H$.
Theorem \ref{theorem:examplesofquotients2} is also deduced from Theorem \ref{theorem:maintheoremprocedural}, and
the fact that, for such $L$, the algebra $\Mat_{s_L}(L)$ appears as a simple factor of $\Q[H]$ for some finite group $H$ and some $s_L$.
The reader may notice that the group
$\mc G$ obtained is always of type $A_n, C_n,$ or $D_n$ but never $B_n, G_2, F_4, E_6, E_7$ or $E_8$.

The above mentioned results open a lot of questions. They show that the classical Torelli group of $\Mod(\Sigma)$ is just
a first in a list of countably many ``generalized Torelli subgroups'' -- $\ker(\rho_{H,r})$ as above. Are these subgroups
(or any of them) finitely generated? Note that these are very different from the ``higher Torelli groups'' (or equivalently, groups in 
the Johnson filtration) obtained from nilpotent quotients of $T_g$; in fact, for any group $\mc I$ in the Johnson filtration and
any $\frak T_{H,r} = \ker(\rho_{H,r})$ for $H$ nontrivial, the product $\mc I \frak T_{H,r}$ is of finite index in $\Mod(\Sigma)$.
(See Section \ref{subsection:johnson}.)

Our theorem can be useful also toward solving the long-standing problem whether 
$\Mod(\Sigma_g)$, for $g \geq 3$, can virtually surject onto $\Z$. In \cite{PutmWiel}, Putman and Wieland showed, roughly speaking, that if
$\rho_{H,p}(\Gamma_{H,p})$ has no finite orbits for all $p$ and $H$ (a collection for which $\ker(p)$ is cofinal would suffice),
then the mapping class group (for a surface of genus $1$ greater) does not virtually surject to $\Z$. 
Theorem \ref{theorem:mainresult} easily implies this condition for $\phi$-redundant $p$. This is just a step in this direction since
the subgroups $\ker(p)$ for all such $p$ do not form a cofinal family. (To get the conclusion that $\Mod(\Sigma)$ does not virtually surject onto $\Z$, one must prove a similar result for covers of a surface with one boundary component. See \cite{PutmWiel} for the precise formulation.)

The key idea for proving Theorem \ref{theorem:mainresult} is that the handlebody subgroup of the mapping class group
behaves like a maximal parabolic subgroup of a semisimple Lie group. Let us elaborate. Along the way, we will indicate where each step
in the argument is proved. As explained above, we are getting a representation
$\rho_{H,p, i}$ of $\Mod(\Sigma)$, or to be more precise a finite index subgroup of it, into $\mc G_{H, i}$.

Let $\Sigma \cong \partial \mc H$ be an identification inducing $\phi: T_g \to \pi_1(\mc H) \cong F_g$ where $\mc H$ is
a genus $g$ handlebody. From \cite[Theorem 5.2]{Jaco},
it follows that all surjective homomorphisms $\phi: T_g \to F_g$ arise this way.
We analyze first the image of $\Map(\mc H)$, the handlebody subgroup of $\Mod(\Sigma)$ for a handlebody $\mc H$ with
boundary $\Sigma$. (More accurately, we will study $\Map(\mc H, \ast)$, the handlebody group with a fixed point.)
This is the subgroup of $\Mod(\Sigma)$ consisting of all isotopy classes containing homeomorphisms which extend to
$\mc H$. This important subgroup of $\Mod(\Sigma)$ has been actively studied in recent years -- see for example \cite{HameHens1, HameHens2}. Via its action
on the fundamental group of $\mc H$, it is mapped onto $\Out(F_g)$, the outer automorphism group of the free group on $g$ generators. Following
carefully the definitions, one sees that when $p: T_g \to H$ is $\phi$-redundant, $\rho_{H,p,i}(\Map(\mc H))$ acts
on a submodule of $M_i$ in precisely the same way as $\rho(\Out(F_g))$ in \cite{GrunLubo} (Section \ref{section:liftrep}). We can therefore appeal to the results of
\cite{GrunLubo} to deduce that $\rho_{H,p,i}(\Map(\mc H))$ contains an arithmetic subgroup of the Levi factor $\mc L$ of a suitable
maximal parabolic subgroup $\mc P = \mc L \mc N^+$ of $\mc G_{H,i}$ (Section \ref{section:liftrep}). 

Moreover, we use again the $\phi$-redundant assumption to show that the image of $\Gamma_{H, p} < \Mod(\Sigma_g, \ast)$
contains nontrivial unipotent elements in the unipotent radical $\mc N^+$ of $\mc P$ as well as of $\mc N^-$, its opposite subgroup (upper triangular
versus lower triangular). At this point, we use both results as well as the fact that $\mc L$ acts irreducibly on $\mc N^+$ and on $\mc N^-$, to deduce
that the image of $\Mod(\Sigma)$ contains finite index subgroups of $\mc U^+(\frak O) = \mc U^+ \cap \mc G(\frak O)$ and 
$\mc U^-(\frak O) = \mc U^- \cap \mc G(\frak O)$ where $\mc U^+$ and $\mc U^-$ are opposite maximal 
unipotent subgroups of $\mc G_{H,i}$ (Section \ref{section:unipotents}). Now, we can appeal to the result of Raghunathan \cite{Ragh} (see also Venkataramana \cite{Venk})
that two such subgroups generate a finite index subgroup of $\mc G_{H,i}^1(\frak O)$ and Theorem \ref{theorem:mainresult} is deduced (Section \ref{section:image}). 

The paper is organized as follows. In Section \ref{section:gaschutz}, we prove an analogue of Gasch\"utz' Theorem for surface groups, and in Section
\ref{section:symplecticgaschutz}, we develop what we call ``Symplectic Gasch\"utz theory'', i.e. identifying the structure of $\hat{R} = \overline{R} \otimes_\Z \Q$,
not merely as a $\Q[H]$-module but also as a module with a skew-Hermitian form over $\Q[H]$ -- an algebra with involution. We relate its structure
to the representation theory of the finite group $H$. In Section \ref{section:tautypekind}, we elaborate on this connection. Most of the material
in Section \ref{section:tautypekind} is likely known to experts on simple algebras, but we chose to include a quick presentation which seems not
to be available in this form. This material is useful when one comes to producing examples of arithmetic quotients out of finite groups $H$
and their representations. But the reader can skip this section on a first reading, going right away to Section \ref{section:isotropic}.

In Section \ref{section:isotropic}, we describe a submodule of $\hat R$ isomorphic to $\Q[H]^2$ which leads (in Section
\ref{section:unipotents}) to an identification of two unipotent elements in $\rho_{H,p,i}(\Gamma_{H,p}$), and we
prove that the sesquilinear form $\langle-,-\rangle$ on $M_i=A_i^{2g-2}$ has an isotropic submodule $M_i'$
isomorphic to $A_i^{g-1}$. The parabolic subgroup $\mc P$ mentioned above consists precisely of those elements of $\mc G_{H,i}$
preserving $M_i'$. Some of the content of Sections \ref{section:liftrep}, \ref{section:unipotents}, and \ref{section:image} have been
indicated above; additionally, Section \ref{section:liftrep} establishes the basic properties of $\rho_{H,p}$, and in
Section \ref{section:image}, we finish the proof of Theorem \ref{theorem:maintheoremprocedural}. 
We end in Section \ref{section:finitegroups} with proofs of the remaining results claimed
in the introduction and some discussion of the arithmetic groups $\Omega_{H, A}$ as $(H, A)$ ranges over all pairs of
finite groups $H$ and simple components $A$ of $\Q[H]$. \\

\subsection{Index of Notation:}
For the convenience of the reader, we collect here some of the more important notation that
is consistent throughout the paper. \\

\noindent $\Mod(\Sigma)$: the mapping class group of the closed surface $\Sigma$\\
$\Mod(\Sigma, \ast)$: the mapping class group of the closed surface $\Sigma$ fixing the point $\ast$\\ 
$g$: genus of the surface \\
$H$ : a finite group \\
$r, r_i$:  an irreducible $\Q$-representation of $H$\\
$A, A_i$: a simple $\Q$-algebra (almost always denoting a component of $\Q[H]$)\\
$M, M_i$: modules over $A, A_i$ (almost always denoting $A^{2g-2}, A_i^{2g-2}$)\\
$L, L_i$: the center of $A, A_i$\\
$\tau$: the canonical involution on $\Q[H]$ and each $A, A_i$ (See Lemma \ref{lemma:algdecomp}.)\\
$K, K_i$: the $\tau$-fixed subfield of $L, L_i$\\
$\langle -,- \rangle$: the sesquilinear pairing on $A_i^{2g-2}$ (or $\hat{R}$; see 
Section \ref{section:pairingdefn}.) \\
$\mc G_{H, i}$: the algebraic group $\Aut_{A_i}(A_i^{2g-2}, \langle -,- \rangle)$ (defined over $K_i$)\\
$\mc G_{H, i}^1$: the elements in $\mc G_{H, i}$ of reduced norm $1$\\
$\frak O, \frak O_i$: an order in $A, A_i$ which is the image of $\Z[H]$\\
$\Omega, \Omega_{H, r}, \Omega_{H, A}$: the arithmetic group consisting of 
	$\frak O, \frak O_i$ points of $\mc G_{H, i}^1$. (where $A = A_i$ and $r$ is the
	representation $\Q[H] \to A$)\\
$T_g := \pi_1(\Sigma_g)$\\
$p: T_g \to H$: some surjective homomorphism\\
$R := \ker(p)$\\
$\overline{R} = R/[R,R]$\\
$\hat{R} = \overline{R} \otimes \Q$\\
$\Gamma_{H,p} = \{\gamma \in \Aut(T_g)^+ = \Mod(\Sigma, \ast) \;\; | \;\; p \circ \gamma = p \}$\\
$\rho_{H, p}: \Gamma_{H,p} \to \Aut_{\Q[H]}(\hat R)$ : the representation for the induced action
of $\Gamma_{H, p}$ on $\hat R$\\
$\rho_{H, p, i}: \Gamma_{H, p} \to  \mc G_{H,i}^1$: the map $\rho_{H, p}$ followed by the projection onto
the action of the $i$th isotypic component of $\hat R$ \\
$\Gamma_H$ :  a finite index subgroup of $\Mod(\Sigma)$ (of some relation to $\Gamma_{H,p}$; see
Section \ref{section:forgetmarkedpt}.)\\
$\rho_{H, r}, \rho_{H, A}$ : alternative names for the representation $\Gamma_H \to \Omega_{H, A}
= \Omega_{H, r}$ where $r: \Q[H] \to A$ is an irreducible representation.\\
$\mc H, \mc H_g$: the genus $g$ handlebody \\
$\Map(\mc H), \Map(\mc H, \ast)$ : the mapping class group of the handlebody (also called 
the handlebody group) respectively without and with a fixed point.\\
$F_g$ : the free group of rank $g$ \\
$\phi: T_g \to F_g$: a surjective homomorphism\\
$p': F_g \to H$: a surjective homomorphism satisfying $p' \circ \phi = p$ (when $p$ is 
$\phi$-redundant)\\
$S := \ker(p')$\\
$\overline{S} = S/[S,S]$\\
$\hat S = \overline{S} \otimes \Q$\\
$\mc P$ : a maximal parabolic in $\mc G_{H, i}$\\
$\mc L$: a Levi factor of $\mc G_{H, i}$\\
$\mc N^+$: the unipotent radical of $\mc P$\\
$\mc N^-$: the opposite subgroup of $\mc N^+$\\
$\mc U^+, \mc U^-$: opposite maximal unipotent radicals of $\mc G_{H, i}$

\subsection{Acknowledgements:} The authors acknowledge useful discussions with Ursula Hamenst\"adt and Sebastian Hensel on the 
handlebody groups and with V. Venkataramana on generators of arithmetic groups. We are also grateful to Andrei Rapinchuk, Eli
Eljadeff, and Uriah First for various discussions on algebras and arithmetic groups.
The work of the 3rd and 4th authors is supported by the ERC, and 
the work of the 2nd and 3rd authors by the NSF and BSF.

\section{A theorem of Gasch\"utz and a generalization to surface groups}
\label{section:gaschutz}
%!TEX root = ./MCGReps.tex
Suppose we have an exact sequence of groups 
$$1 \to R \to T \overset{p}{\to} H \to 1$$ 
where $T$ is the fundamental group of a closed orientable surface of genus $g$ and $H$ is a finite group. The action of $T$ on $\overline{R} = R/ [R, R]$
by conjugation descends to an action of $H$. Thus, $\overline{R}$ has the structure of a $\Z[H]$ module. In this section, 
we prove Proposition \ref{prop:surfacegaschutz} from the introduction. We recall the proposition here.

\surfgasc*

As mentioned in the introduction, this is a known result due to Chevalley--Weil.
We adopt a topological viewpoint to give an alternative proof of Proposition \ref{prop:surfacegaschutz}, and 
we begin by translating algebraic objects into topological ones. For any normal finite index subgroup $R < T$,
there is a corresponding finite index regular cover $\tilde \Sigma \to \Sigma$ such that the image 
$\pi_1(\tilde \Sigma) \to \pi_1(\Sigma)$ is precisely $R$. Using this, one can identify $\HH_1(\tilde \Sigma, \Z)$
with $\overline{R}$ and $\HH_1(\tilde \Sigma, \Q) \cong \hat R$. The group $H = T/R$ acts on the cover
by deck transformations and thereby induces an action of $H$ on $\HH_1(\tilde \Sigma, \Q)$. The isomorphism 
$\HH_1(\tilde \Sigma, \Q) \cong \hat R$ is an isomorphism of $\Q[H]$-modules.

\subsection{Gasch\"utz's Theorem}
In the case where $T$ is replaced by a free group $F$, the description of $\hat{R}$ is a classical result of Gasch\"utz.
Since we require its use in later sections, we present Gasch\"utz's theorem. We also provide a new topological
proof of the theorem which we then adapt for the analagous theorem for surfaces.

\begin{theorem}{(Gasch\"utz)} \label{theorem:gaschutz}
Suppose $\displaystyle 1 \to R \to F_n \overset{p}{\to} H \to 1$ is a short exact sequence where $F_n$ is the free group on $n$ generators
and $H$ is a finite group. Let $\overline{R} = R/[R,R]$ and $\hat{R} = \overline{R} \otimes_\Z \Q$. Then, there is a $\Q[H]$-module isomorphism:
$$\hat{R} \cong \Q[H]^{n-1} \oplus \Q$$
\end{theorem}

\begin{proof}
To prove this, let us first identify $F_n$ with the fundamental group of an $n$-petalled rose $Y$ with oriented edges where each edge
is one of the free generators $x_i$ of $F_n$. Let $\widetilde{Y} \to Y$ be the cover corresponding to $R$. 
Lift the orientation of $Y$ to an orientation of the edges of $\widetilde{Y}$.

We compute $\HH_1(\widetilde{Y}, \Q)$ via cellular homology. Let $C_i(\widetilde{Y}, \Q)$ denote formal sums with $\Q$ coefficients
of $i$-cells of $\widetilde{Y}$.
Since there are no $2$-cells, $\HH_1(\widetilde{Y}, \Q)$ is the kernel of the boundary map $\partial_1$. Pick some vertex
$\tilde \ast$ of the graph $\widetilde{Y}$, and let $e_1, \dots, e_n$ be the (oriented) edges going out from $\tilde \ast$
where $e_i$ covers the edge corresponding to the generator $x_i$ of $F_n$. 
The group $H$ acts freely on the orbit of $e_i$, the $H$-orbit of $e_i$ and $e_j$ are disjoint if $i \neq j$,
and the $H$-orbits of all the $e_i$ cover $\widetilde{Y}$.
Thus, we have an internal $\Q[H]$-module direct sum decomposition of the space of $1$-chains
$$C_1(\widetilde{Y}, \Q) = \bigoplus_{i=1}^n \Q[H] \cdot e_i \cong \Q[H]^n$$
Furthermore, $C_0(\widetilde{Y},  \Q) = \Q[H] \cdot \tilde \ast \cong \Q[H]$. The boundary map 
is a $\Q[H]$-homomorphism, and furthermore, 
$\partial_1(e_i) = (h_i - 1) \cdot \tilde \ast$ where $h_i = p(x_i)$.

The above argument shows that the image of $\partial_1$ lies in $\frak{h} \cdot \tilde \ast$ where $\frak{h}$ is the
augmentation ideal; i.e. $\frak{h}$ is the kernel of the augmentation map $\epsilon: \Q[H] \to \Q$ defined by $\sum_{h \in H} \alpha_h h \mapsto \sum_{h \in H} \alpha_h$. 
Since we know that $\dim_{\Q}(\HH_0(\widetilde{Y}, \Q)) = 1$, the image of
$\partial_1$ must be the entire augmentation ideal. Semi-simplicity of $\Q[H]$-modules implies
$$\Q[H]^n \cong C_1(\widetilde{Y}, \Q) \cong \HH_1(\widetilde{Y}, \Q) \oplus \frak{h}.$$
Note that $\epsilon$ is a $\Q[H]$-module homomorphism from $\Q[H]$ to the trivial module so $\Q[H] \cong \Q \oplus \frak{h}$
and $\HH_1(\widetilde{Y}, \Q) \cong \Q[H]^{n-1} \oplus \Q$.
\end{proof}

\subsection{Theorem for surface groups}
We now prove Proposition \ref{prop:surfacegaschutz}. Let $Y_1$ be the $2g$-petalled rose
with oriented loops, and label the loops by $a_i, b_i$ for $i = 1, \dots g$.
Let $Y$ be the $2$-dimensional CW-complex obtained by gluing a $2$-cell along its boundary to the path
$[a_1, b_1] [a_2, b_2] \dots [a_g, b_g]$ in $Y_1$ where $[x, y] = xy x^{-1} y^{-1}$. It is well known
that $Y$ is a closed, genus $g$ surface. Let $\widetilde{Y}$ be the cover corresponding
to the subgroup $R < T = \pi_1(Y)$.

As in the above proof, $\partial_1(C_1(\widetilde{Y}, \Q)) = \frak{h}$ still holds. However,
in this case, $\partial_2(C_2(\widetilde{Y}, \Q))$ is nontrivial, so we must determine it to compute
$\HH_1(\widetilde{Y}, \Q)$. The group $H$
acts freely and transitively on the $2$-cells of $\widetilde{Y}$, so $C_2(\widetilde{Y}, \Q) \cong \Q[H] \cdot c_2 \cong \Q[H]$
where $c_2$ is some oriented $2$-cell of $\widetilde{Y}$. 
By semi-simplicity then,
finding $\partial_2(C_2(\widetilde{Y}, \Q))$  is equivalent to determining $\ker(\partial_2)$ which, since there are no $3$-cells, is
$\HH_2(\widetilde{Y}, \Q)$. Since $\widetilde{Y}$ is a surface,  
$\HH_2(\widetilde{Y}, \Q) \cong \Q$ as a $\Q$-vector space. However, we must verify that the $\Q[H]$-module structure
is trivial. The action of $h \in H$ on $\HH_2(\widetilde{Y}, \Q)$ is multiplication by the degree of the map,
and since $H$ acts by orientation-preserving homeomorphisms, that degree is necessarily $1$.

Thus, the image of
$\partial_2$ is isomorphic to $\frak{h}$, and by semi-simplicity of $\Q[H]$-modules,
$$\Q[H]^{2g} \cong C_1(\widetilde{Y}, \Q) \cong \HH_1(\widetilde{Y}, \Q) \oplus \frak{h}^2,$$
and the desired result follows. \eop

\section{Symplectic Gasch\"utz Theory}
\label{section:symplecticgaschutz}
%!TEX TS-program = ./Scripts/wrapper
Let us fix $R < T$, a finite index normal subgroup, and set $H = T/R$ and $\hat R = (R/[R,R]) \otimes_\Z \Q$ as above.
For surface groups $T$, the $\Q[H]$-module $\hat R$ has a richer structure than in the analogous situation for free
groups. Specifically, $\hat R$ admits a natural $\Q[H]$-valued sesquilinear form $\langle -, - \rangle$. In this section, we describe the form,
and we exhibit a decomposition of the pair $(\hat R, \langle -, - \rangle)$ into factors. Later in Section \ref{section:liftrep}, we will see 
that this structure is preserved by the action of $\rho_{H,p}(\Gamma_{H,p})$ (Lemma \ref{lemma:liftpreserveform}).

\subsection{The sesquilinear pairing on $\hat{R}$} \label{section:pairingdefn}

We first define a few necessary terms. An {\it anti-homomorphism} $\tau: A \to A$ of a $\Q$-algebra is a $\Q$-linear map of $A$ such that
$\tau(a_1 a_2) = \tau(a_2) \tau(a_1)$ for all $a_1, a_2 \in A$; furthermore $\tau$ is an {\it involution} if $\tau^2 = \Id$. 
Suppose $M$ is an $A$-module. A form $\langle -, - \rangle: M \times M \to A$ is {\it sesquilinear} (relative to the involution $\tau$)
if it is $\Q$-bilinear and for any $r, s \in A$, $m, m' \in M$
$$\langle r m, s m' \rangle = r \langle m, m' \rangle \tau(s)$$
The form furthermore is {\it skew-Hermitian} if $\langle m, m' \rangle = - \tau( \langle m', m \rangle)$ and 
{\it nondegenerate} if for all nonzero $m \in M$, there is an $m' \in M$ such that $\langle m, m' \rangle \neq 0$. 

Any group ring $\Q[H]$ admits a canonical involution $\tau$ defined by setting $\tau(h) = h^{-1}$ for $h \in H$ and extending linearly.
Recall that $\hat{R}$ is naturally identified with $\HH_1(\tilde \Sigma, \Q)$ which has an alternating intersection form, and we denote
this form by $\langle -, - \rangle_{\Sp}$.
In a similar way to \cite{Reide}, \cite[Section 3]{Hemp}, we define the following $\Q[H]$-valued form on $\hat{R}$:
\begin{equation} \label{eqn:sesqform} \langle x, y \rangle = \sum_{h \in H} \langle x, h y \rangle_{\Sp} h. \end{equation}

\begin{lemma}
The form $\langle -, - \rangle$ is nondegenerate, sesquilinear with respect to $\tau$, and skew-Hermitian. \end{lemma}
\begin{proof}
Nondegeneracy follows from the nondegeneracy of the symplectic form. It can be readily checked that sesquilinearity follows
from the fact that $H$ preserves the symplectic form. The action of $H$ preserves the symplectic form as it is equivalent to the
action of deck transformations on $\HH_1( \tilde \Sigma, \Q)$ and deck transformations act by orientation preserving
homeomorphisms. The form $\langle -, - \rangle $ is skew-Hermitian since $\langle -, - \rangle_{\Sp}$ is alternating.
\end{proof}

\subsection{Simple components of $\Q[H]$}
We now recall some basic facts about the group ring $\Q[H]$ and its modules. The ring $\Q[H]$ is a semisimple $\Q$-algebra and thus
is isomorphic to a finite product of simple $\Q$-algebras
$$\Q[H] = \Q \times \prod_{i=1}^\ell A_i.$$
Moreover, for each $i$, we have $A_i \cong \Mat_{m_i}(D_i)$ for some finite-dimensional division algebra $D_i$ with center $L_i$
which is a finite-dimensional field extension of $\Q$. 
For all $i$, the algebra $A_i$ acts on the $\Q$-vector space $V_i = D_i^{m_i}$ by left multiplication; via the projection to $A_i$,
each $V_i$ is a $\Q[H]$-module which is furthermore irreducible. We say $V_i$ is the irreducible module (or representation) corresponding to $A_i$.
Moreover, every irreducible $\Q[H]$-module (or representation of $H$) is isomorphic to one of the $V_i$, and note that as a $\Q[H]$-module, $A_i \cong V_i^{m_i}$.

We describe how $\tau$ acts on the decomposition in the following lemma.

\begin{lemma} \label{lemma:algdecomp}
Each $A_i$ in the decomposition of $\Q[H]$ is $\tau$-invariant.
\end{lemma}
\begin{proof}
Note that each 
$A_i$ is a minimal two-sided ideal of $\Q[H]$. For the purpose of the proof, set $A_0$ to be the trivial factor $\Q$ in the decomposition of $\Q[H]$ . 
Since $\tau$ is an anti-homomorphism, $\tau$ sends minimal two-sided ideals to minimal two-sided ideals.
Since $\tau$ is order $2$, there are two possibilities; for any $i$ either $\tau(A_i) = A_i$, or there
is a $j \neq i$ such that $\tau(A_i) = A_j$ and $\tau(A_j) = A_i$. We must rule out the second possibility. (Our proof depends on the
base field being the totally real field $\Q$. The lemma is false, for instance, over $\C$.)

Let $e_i$ be the unit in $A_i$. It suffices to show that $\tau(e_i) A_i \neq 0$. Consider the representation
$$\Psi: \Q[H] \to \End_{\Q}(A_i) \cong \Mat_t(\Q)$$
given by left multiplication where $t = \dim_\Q(A_i)$. Viewing the matrices with entries in $\C$ yields a representation of $H$ to $\GL_t(\C)$, and
so $\chi(h) = \overline{\chi(h^{-1})}$ where $\chi$ is the character for $\Psi$ and $\overline{\phantom{a} }$ indicates complex conjugation. Since the traces are necessarily rational, in fact
$\chi(h) = \chi(h^{-1})$, and linearity of $\chi$ implies $\chi(r) = \chi(\tau(r))$ for any $r \in \Q[H]$. Thus, $0 \neq \chi(e_i) = \chi(\tau(e_i))$
and left multiplication on $A_i$ by $\tau(e_i)$ is non-zero.
\end{proof}

Recall from Theorem \ref{theorem:gaschutz} that $\hat{R} \cong \Q[H]^{2g-2} \oplus \Q^2$. Thus, as a $\Q[H]$-module
$$ \hat{R} \cong \Q^{2g} \oplus (\bigoplus_{i=1}^m A_i^{2g-2})$$
Let $M_i \subseteq \hat{R}$ be the submodule such that $M_i \cong A_i^{2g-2}$. Each $A_i$ is isomorphic to several copies
of an irreducible $\Q[H]$-module $V_i$ and $V_i \ncong V_j$ for $i \neq j$. Hence, the submodule $M_i$ is unique, and furthermore
any $\Q[H]$-automorphism $\varphi$ of $\hat{R}$ restricts to an automorphism of each $M_i$. Since all $A_j$ except $A_i$ annihilate $M_i$, we obtain a representation
$$ \Aut_{\Q[H]}(\hat R, \langle-,-\rangle) \to \Aut_{A_i}(M_i, \langle-,- \rangle).$$
Our next task then is to understand this automorphism group. (Note that we omitted the module corresponding to the trivial representation.
This is because the automorphism group turns out to be $\Sp(2g, \Q)$ and the representation $\Gamma_{H,p} \to \Sp(2g, \Q)$ is the standard 
symplectic representation of the mapping class group.)

\subsection{The automorphism group of $M_i$} \label{subsection:AutM}
We first need to know the properties of $\langle -, - \rangle$ restricted to each $M_i$. By abuse of notation,
we will view $\tau$ as an involution on $A_i$.
\begin{lemma}
For any $m, m' \in M_i$, the pairing $\langle m, m' \rangle$ lies in $A_i$. Furthermore, $\langle - , - \rangle: M_i \times M_i \to A_i$ is a nondegenerate, sesquilinear,
skew-Hermitian form when viewed as a form with values in $A_i$.
\end{lemma}
\begin{proof}
Let $e_i$ be the unit in $A_i$. When restricted to $M_i$, the form takes values in $A_i$ since sesquilinearity implies that 
$$\langle m, m' \rangle = \langle e_i m, e_i m' \rangle = e_i \langle m, m' \rangle \tau(e_i) \in A_i$$
When viewed as a form with values in $A_i$, it is clear that the form remains sesquilinear and skew-Hermitian. Since
$\langle-, -\rangle$ is nondegenerate on all of $\hat R$, to show that it is nondegenerate on $M_i$, 
we only need to show that the $M_i$ are mutually perpendicular.
This is verified by essentially the same computation. If $m \in M_i$ and $m' \in M_j$ and $i \neq j$, then 
$$\langle m, m' \rangle = \langle e_i m, e_j m' \rangle = e_i \langle m, m' \rangle \tau(e_j)$$
Lemma \ref{lemma:algdecomp} implies that $\tau(e_j) = e_j$, so the product is $0$.
\end{proof}

The group $\Aut_{A_i}(M_i, \langle -, - \rangle)$ is the set of the $K_i$-points of an algebraic group defined over $K_i$ where
$K_i = L_i^\tau$, the fixed field of $\tau$ acting on $L_i$. Our next goal is to determine the structure of this algebraic group; to this end, we will use
some of the basic theory of involutions and follow the exposition in \cite[Section 2.3.3]{PlatRapi}.

\paragraph{Involutions}
In this subsection, we utilize the dictionary between nondegenerate bilinear forms and involutions to describe the structure of 
$\mc G_{H,i} =\Aut_{A_i}(M_i, \langle-,-\rangle)$. For convenience, we will fix $A = A_i$,
$M = M_i$, and $\mc G = \mc G_{H,i}$ as we will only be considering each algebra individually, so $M \cong A^{2(g-1)}$ with the sesquilinear form as described
above. We set $L$ to be the center of $A$ and $K = L^\tau$ the subfield of $L$ fixed by $\tau$. 
The dictionary is that the sesquilinear form goes over
to the unique involution $\sigma: \End_A(M) \to \End_A(M)$ satisfying the following
for all $m, m' \in M$ and $C \in \End_A(M)$:
$$ \langle Cm, m' \rangle = \langle m, \sigma(C)m' \rangle.$$
This is the {\it adjoint involution} associated to the sesquilinear form.
Now,
$$\begin{array}{rcl} \Aut_A(M, \langle-,-\rangle) & = & \{ C \in \End_A(M) \; | \; \langle C m, Cm' \rangle = \langle m, m' \rangle \;\; \forall m, m' \in M\} \\
									& = & \{ C \in \End_A(M) \; | \; \langle m, \sigma(C) Cm' \rangle = \langle m, m' \rangle \;\; \forall m, m' \in M\} \\
									& = & \{ C \in \End_A(M) \; | \; \sigma(C) C = \Id\} = \mc G \end{array} $$
In other words, those automorphisms preserving the form can be entirely determined by the involution alone.

Involutions of finite-dimensional simple algebras over a field $k$ form a trichotomy based on kind/type. The first distinction is that of kind.
An involution $\tau$ on a finite-dimensional simple algebra $A$ over a field $k$ is {\it of the first kind} if it fixes the center of $A$ and
is {\it of the second kind} if it does not. Furthermore, if $A$ is of the first kind and $n^2 = \dim_L(A)$, then either 
$\dim_L(A^\tau) = n(n+1)/2$ or $\dim_L(A^\tau) = n(n-1)/2$ where $L$ is the center of $A$ and $A^\tau$ is the subspace of $A$ 
fixed by $\tau$. If the former is true, we say $\tau$ is of {\it orthogonal type}
and if the latter is true, $\tau$ is of {\it symplectic type}. It is a standard fact that these are the only two possible types of involutions
of the first kind and this fact is, in
particular, implied by Lemma \ref{lemma:changeinv}. We note that in \cite{PlatRapi}, orthogonal and symplectic type are called
first and second type respectively.

We now set $B = \End_A(M)$ and seek to understand the kind and type of the involution $\sigma$ in terms of the kind and type of $\tau$. First,
however, we prove that $B$ is itself a simple algebra. Since $M$ is a free $A$-module of rank $2g-2$, $B$ is isomorphic to $\Mat_{2g-2}(A^{op})$ and we henceforth
fix some such identification.
\begin{lemma}
$B$ is simple.
\end{lemma}
\begin{proof}
Since $A$ is a simple finite dimensional $\Q$-algebra, so is $A^{op}$, and so $A^{op}$ is isomorphic to a matrix algebra $\Mat_m(D)$ over a division ring $D$. 
Thus, $B \cong \Mat_{2g-2}(A^{op}) \cong \Mat_{(2g-2)m}(D)$ which is simple.
\end{proof}

We now convert our terminology into that of matrices to establish the relation between the type and kind of $\tau$ and $\sigma$. 
Note that $\tau$ is also an involution on $A^{op}$ of the same kind and type. For a matrix $C = (c_{ij}) \in \Mat_{2g-2}(A^{op}) = B$, we set $C^* = (\tau(c_{ji}))$. 
Letting $e_1, \dots, e_{2g-2}$ be the standard free basis of $A^{2g-2}$, we define a matrix $F = (f_{ij})$ where 
$f_{ij} := \langle e_j, e_i \rangle$. 
It can be checked that $\langle Cv, w \rangle = \langle v, Dw \rangle$ for all $v,w \in M$ if and only if $FC = D^*F$. Moreover, 
$F^* = -F$ as $\langle-,-\rangle$ is skew-Hermitian. Thus, since $(F C)^* = C^* F^* = - C^* F$ and $(D^* F)^* = - F D$, we have
$$\sigma(C) = F^{-1} C^* F.$$

\begin{lemma} \label{lemma:oppositetype}
The involutions $\sigma$ and $\tau$ are of the same kind. The type of $\sigma$ is opposite that of $\tau$.
\end{lemma}
\begin{proof}
It is clear that $L^\sigma = L^\tau$ so they are of the same kind. A straightforward argument counting dimension shows that the involution $*$
on $B$ has the same type as $\tau$ if they are both first kind. Because $\langle-,-\rangle$ is skew-Hermitian, it follows that $F^* = -F$. 
Then, a computation shows that if $C^* = - C$, then $\sigma(F^{-1} C) = F^{-1} C$
and if $C^* = C$, then $\sigma(F^{-1} C) = -F^{-1} C$. Hence, via $F^{-1}$, the $+1$ eigenspace of $*$ is isomorphic to the
$-1$ eigenspace of $\sigma$ and vice versa. This implies $\sigma$ has the opposite type to $*$.
\end{proof}

\paragraph{Extending scalars}
Consider now the following three examples of involutions $\nu$ for matrix algebras over
$\C$.
\begin{itemize}
\item[(1)] $\nu: \Mat_n(\C) \to \Mat_n(\C)$ defined by $\nu(N) = N^t$.
\item[(2)] $\nu: \Mat_n(\C) \to \Mat_n(\C)$ defined by $\nu(N) = J N^t J^{-1}$ where $n$ is even and $J = \left( \begin{array}{cc} 0 & I \\ -I & 0 \end{array} \right)$.
\item[(3)] $\nu: \Mat_n(\C) \times \Mat_n(\C) \to \Mat_n(\C) \times \Mat_n(\C)$ defined by $\nu(A, B) = (B^t, A^t)$.
\end{itemize}
Note that if we consider the elements satisfying $\nu(M) M = I$, then from (1), we obtain $\OO_n(\C)$, from (2), $\Sp_n(\C)$ and from (3), $\GL_n(\C)$.
Moreover, it can be easily checked that the involution in (1) is of orthogonal type and the involution in (2) is of symplectic type.
We quote the following result \cite{PlatRapi} which essentially tells us these are the only involutions after extending scalars.

\begin{lemma} \label{lemma:changeinv}
Let $B$ be a finite-dimensional simple $\Q$-algebra with center $L$ and involution $\sigma$. Let $K = L^\sigma$ and
let $\tilde{\sigma}$ be the unique $\C$-linear extension of $\sigma$ to $\displaystyle B \underset{K}{\otimes} \C$. Set
$n = \dim_L(B)$.
\begin{itemize}
\item If $\sigma$ is of the first kind, then there is a $\C$-algebra isomorphism 
$$\varphi: B \underset{K}{\otimes} \C \cong \Mat_n(\C)$$
such that $\nu = \varphi \tilde \sigma \varphi^{-1}$ is one of the involutions (1) or (2).
\item If $\sigma$ is of the second kind, then there is a $\C$-algebra isomorphism 
$$\varphi: B \underset{K}{\otimes} \C \cong \Mat_n(\C) \times \Mat_n(\C)$$
such that $\nu = \varphi \tilde \sigma \varphi^{-1}$ is the involution (3).
\end{itemize}
\end{lemma}

At this point, we can already prove part of Theorem \ref{theorem:Gtype}. Specifically, we can show that the type of $\mc G$
is determined by the type and kind of $\tau$. We will be able to prove the full theorem after establishing further results
in Section \ref{section:tautypekind}.

\subsection{Reduced norm, reduced trace, and $\mc G^1(\frak O)$} \label{subsection:gofinindex}

We conclude this section with a few facts on reduced norms and apply them to 
$$\mc G(\frak O) = \Aut_{\frak O}(\frak O^{2g-2}, \langle -,-\rangle) \subseteq \End_A(A^{2g-2})$$
where, as defined in the introduction, $\frak O$ is the image of $\Z[H]$ in $A$. Our goal is to prove Proposition \ref{prop:GOfiniteindex}, which 
gives an ``upper bound'' on the image of the representation $\rho$ which we study (while Theorem \ref{theorem:mainresult}
gives the ``lower bound''). Recall from the introduction that $\mc G^1(\frak O) = \mc G_{H,i}^1(\frak O)$ is defined to be
the subgroup of $\mc G(\frak O)$ consisting of elements of reduced norm $1$.

The reduced norm (over $L$) of a finite-dimensional central simple $L$-algebra $B$ is defined as follows. Let $E$ be
any field extension of $L$ such that $B \otimes_L E$ {\em splits}, i.e. there is some isomorphism 
$\varphi: B \otimes_L E  \to \Mat_n(E)$ for some $n$. Then the {\em reduced norm} (over $L$) for $b \in B$ is 
$$\nrd_{B/L}(b) = \det(\varphi(b\otimes 1)).$$
The reduced norm is independent of the isomorphism $\varphi$, and $\nrd_{B/L}(B) \subseteq L$. (See e.g. \cite[Section 9]{Rein}.) 
In later sections, we also use the {\em reduced trace} which, for $b \in B$, is
$$\trd_{B/L}(b) = \Tr(\varphi(b\otimes 1))$$
where $\Tr$ is the trace of the $E$-linear map $\varphi(b)$.
Just as for reduced norm, the reduced trace always lies in $L$ and is independent of $\varphi$. 
From the definitions, it is clear that $\trd_{B/L}$ is $L$-linear.

\begin{lemma} \label{lemma:rednrminG}
Let $B = \Mat_{2g-2}(A^{op})$. Let $b \in \mc G = \Aut_A(A^{2g-2}, \langle-,-\rangle) \subseteq B$ and $\lambda = \nrd_{B/L}(b)$. Then, $\tau(\lambda) \lambda = 1$. 
In particular, if $\tau$ is of the first kind, then $\lambda = \pm 1$.
\end{lemma}
\begin{proof}
We know that $b \in \mc G$ satisfies $1 = \sigma(b) b$, and so 
$$1 = \nrd_{B/L} (\sigma(b) )\nrd_{B/L}(b) = \nrd_{B/L}(\sigma(b)) \lambda.$$
Thus, it suffices to show that $\nrd_{B/L}(\sigma(b)) = \tau(\lambda)$.

Let $E$ be a finite Galois field extension of $L$ such that there is an isomorphism $\varphi: B \otimes_L E \cong \Mat_n(E)$ for some $n$.
Let $\phi$ be the composition $B \to B \otimes_L E \cong \Mat_n(E)$.
Let $\nu = \tau|_L$, and let $\tilde \nu$ be an extension of $\nu$ to a Galois automorphism of $E$. For $C = (c_{ij}) \in \Mat_n(E)$,
define $C^* = ( \tilde \nu(c_{ji}))$. Now, let $\psi: B \to \Mat_n(E)$ be the map defined by 
$ b \mapsto (\phi(\sigma(b)))^*$. Note that since $\sigma$ and $^*$ are both anti-homomorphisms, $\psi$ is a homomorphism,
and moreover, one can check that $\psi$ is an $L$-algebra homomorphism. By the Skolem-Noether Theorem, $\psi$
and $\phi$ are the same up to conjugation by an element of $\Mat_n(E)$. Thus,
$$ \nrd_{B/L}(\sigma(b)) = \det(\phi(\sigma(b))) = \nu^{-1} ( \det( \psi(b)  )) = \nu^{-1} (\det(\phi(b)))= \nu^{-1}(\nrd_{B/L}(b)).$$
Since $\nrd_{B/L}(b) \in L$, we can replace $\nu^{-1}$ with $\tau^{-1}$ or equivalently $\tau$.
\end{proof}

We next show that elements in the ring of integers in a cyclotomic field $E$ of ``absolute value'' $\pm 1$ are roots of unity.
\begin{lemma} \label{lemma:normtortunity}
Let $E = \Q(\zeta)$ be a cyclotomic field, let $\mc O_E$ its ring of integers, and let $\iota: E \to \C$ be some embedding
into $\C$. If $\lambda \in \mc O_E$ and $| \iota(\lambda)|^2 = 1$, then $\lambda$ is a root of unity.
\end{lemma}
\begin{proof}
According to a theorem of Kronecker, if all the Galois conjugates of an algebraic integer $\lambda$ have 
absolute value  $\leq 1$, then $\lambda$ is a root of unity. Let $\nu$ be the order $2$ Galois automorphism
that is the restriction of complex conjugation (via $\iota$). As $|\iota(e)|^2 = \iota(\nu(e) e)$ for any $e \in E$,
it will suffice to show that $\nu(\sigma(\lambda)) \sigma(\lambda) = 1$ for all Galois automorphisms $\sigma$
of $E$. Since $E/\Q$ is cyclotomic, its Galois group is abelian, and thus 
$$\nu(\sigma(\lambda)) \sigma(\lambda) = \sigma(\nu(\lambda)) \sigma(\lambda) = \sigma(\nu(\lambda) \lambda) = \sigma(1) = 1.$$
\end{proof}

\begin{remark}
The above lemma holds for any CM-field. The essential point is that complex conjugation always restricts
to the same automorphism of $E$ regardless of the embedding of $E$ into $\C$.
\end{remark}

\begin{prop} \label{prop:GOfiniteindex}
The group $\mc G^1(\frak O)$ is of finite index in $\mc G(\frak O)$.
\end{prop}
\begin{proof}
It suffices to show that the image of $\mc G(\frak O)$ under $\nrd_{B/L}$ consists only of roots of unity. If $\tau$ is of the first kind,
this is obvious by Lemma \ref{lemma:rednrminG}. Suppose $\tau$ is of the second kind.
Any $\lambda \in \nrd_{B/L}(\mc G(\frak O))$ lies in 
$\mc O_L$ where $\mc O_L$ is the ring of integers in $L$, and by Lemma \ref{lemma:rednrminG},
$\lambda$ satisfies $\tau(\lambda) \lambda = 1$.
Fix some embedding $\iota : L \to \C$. 
By Proposition \ref{prop:typefromrealrep} below, $\tau|_L$ extends to complex conjugation, and so 
$1= \iota(\tau(\lambda) \lambda) = \overline{\iota(\lambda)} \iota(\lambda) = |\iota(\lambda)|^2$.
Since $L$ is a subfield of a cyclotomic field $E$ (to which $\iota$ extends), 
by Lemma \ref{lemma:normtortunity}, $\lambda$ is a root of unity in $L$.
\end{proof}

\section{Determining type and kind of $\tau$}
\label{section:tautypekind}

The kind of algebraic group into which $\rho_{H, p,i}$ maps depends on $\tau|_{A_i}$. 
The kind and type of $\tau$ depends on the representation theory of the finite group $H$. We therefore dedicate this section to study
this dependency. This will be important later when we come in Section \ref{section:finitegroups} to produce specific arithmetic groups as
(virtual) quotients of $\Mod(\Sigma)$. However, this section may be skipped on a first reading if the reader wants to see first why the image of 
$\rho_{H,p,i}$ is an arithmetic group. Before beginning, we note additionally that the material presented here is known to experts,
but we include it for the convenience of the non-expert and for lack of a convenient reference. 
As in the previous section, we will focus on one nontrivial component of $\Q[H]$ and drop the subscript $i$. Since we fix some $A$, by abuse of notation, we will
refer to $\tau|_A$ as $\tau$.

Recall that there is a trichotomy: $\tau$ can be of the first kind and orthogonal type, of the first kind and symplectic type, or of the second kind. This trichotomy corresponds in a nice
way to other trichotomies in the representation theory of the finite group $H$.
The first trichotomy arises from the various kinds of invariant bilinear forms for complex representations. Fix some embedding $L \hookrightarrow \C$
and some isomorphism $A \otimes_L \C \cong \Mat_n(\C)$. Then the composition $H \to A \to A \otimes_L \C \to \Mat_n(\C)$ yields 
an action of $H$ on $V = \C^n$ which in turn is an irreducible $\C[H]$-module. There are three possibilities: $V$ has an $H$-invariant
non-degenerate symmetric bilinear form, an $H$-invariant non-degenerate alternating bilinear form, or no nonzero $H$-invariant bilinear form.
This corresponds to the type and kind of $\tau$ as follows.

\begin{prop} \label{prop:typefrominvtform}
Keeping notation as above, we have that, restricted to $A$, the involution $\tau$ is
\begin{itemize}
\item of the first kind and orthogonal type if and only if $V$ has an $H$-invariant nondegenerate symmetric bilinear form,
\item of the first kind and symplectic type if and only if $V$ has an $H$-invariant nondegenerate alternating bilinear form,
\item of the second kind if and only if $V$ has no invariant non-zero $H$-invariant nondegenerate bilinear form.
\end{itemize}
\end{prop}

Before proving this proposition, we first describe two other trichotomies closely related to the trichotomy of kind and type.
One of these is a trichotomy involving the real representations of $H$. This trichotomy also relates to a dichotomy
involving the center $L$ of $A$. In order to present the dichotomy, we first recall some elementary facts about $L$.
A short proof of the lemma is given below.

\begin{lemma} \label{lemma:cyccenter}
Let $A$ be a simple component of $\Q[H]$ with center $L$. Then, $L$ is a subfield of a cyclotomic field extension of $\Q$.
Moreover, $L$ is either totally imaginary (there are no embeddings $L \hookrightarrow \R$) or
totally real (all embeddings $L \hookrightarrow \C$ have image in $\R$).
\end{lemma}
 
The next proposition presents the relation to kind and type. The isomorphisms
in the proposition below are as $\R$-algebras.

\begin{prop} \label{prop:typefromrealrep}
Let notation be as above, let $n^2 = \dim_L(A)$ and let $e: L \hookrightarrow \C$ be an arbitrary embedding. Then, the embedding
satisfies $K = e^{-1}(\R)$, i.e. $\tau$, restricted to $A$, is of the first kind if and only if $L$ is totally real.
If $\tau$ is of the second kind, then $\tau|_L$ extends to complex conjugation via the embedding $e$.
Moreover, restricted to $A$, the involution $\tau$ is
\begin{itemize}
\item of the first kind and orthogonal type if and only if $A \otimes_{K,e} \R \cong \Mat_n(\R)$,
\item of the first kind and symplectic type if and only if $A \otimes_{K,e} \R \cong \Mat_{n/2}(\bbH)$,
\item of the second kind if and only if $A \otimes_{K,e} \R \cong \Mat_n(\C)$.
\end{itemize}
\end{prop}

Another trichotomy relates to the Frobenius-Schur indicator $\iota \chi$ for any character $\chi$ of $H$, defined to be
$$ \iota \chi = \frac{1}{|H|}\sum_{h \in H} \chi(h^2).$$
We will be quoting some results from \cite{JameLieb} where $\iota\chi$ is defined differently and only for irreducible $\C$-characters (over $\C$). However,
the definitions coincide for irreducible characters (see the proof of Theorem 23.14 in \cite{JameLieb}), and this is all we require.

\begin{prop} \label{prop:typefromchi}
Continuing our notation as above, let $\chi$ be the character corresponding to the irreducible $\Q$-representation $V$ of $H$. Then, restricted to $A$, the involution
$\tau$ is
\begin{itemize}
\item of the first kind and orthogonal type if and only if $\iota \chi > 0$,
\item of the first kind and symplectic type if and only if $\iota \chi < 0$,
\item of the second kind if and only if $\iota \chi = 0$
\end{itemize}
\end{prop}

We start with the following elementary lemma showing that $\tau$ is the unique involution on $A$ with a certain property.
\begin{lemma} \label{lemma:unique}
Suppose $k$ is some field and $\psi: k[H] \to A$ is a surjective homomorphism of $k$-algebras. Then, there is at most one $k$-algebra 
involution $\nu$ on $A$ such that $\nu(\psi(h)) \psi(h) = \Id$.
\end{lemma}
\begin{proof}
The elements $\psi(h)$ generate $A$ as a $k$-vector space.
\end{proof}

\begin{proof}[Proof of Proposition \ref{prop:typefrominvtform}]
Suppose $\tau$ is of the first kind.
Then, $\tau$ is $L$-linear and the involution extends to a ($\C$-algebra) involution $\tilde \tau$ of $A \otimes_L \C$. By Lemma \ref{lemma:changeinv}, there
is an isomorphism $\varphi: A \otimes_L \C \cong \Mat_n(\C)$ such that $\nu = \varphi \tilde \tau \varphi^{-1}$ is the standard symplectic
or orthogonal involution. If $\psi$ is the composition $\Q[H] \to A \to A \otimes_L \C$, then it is the case that $\tilde \tau(\psi(h)) \psi(h) = \Id$ and so
$\nu(\varphi(\psi(h))) \; \varphi(\psi(h))  = \Id$. This implies $\varphi(\psi(h))$ preserves
a nondegenerate symmetric bilinear form (resp. alternating form) if $\tau$ is of orthogonal (resp. symplectic) type. 

Suppose that $V$ has a non-zero $H$-invariant bilinear form. Then, since $V$ is irreducible, the form must be nondegenerate
because a nontrivial null-space would necessarily be $H$-invariant. Now, let $\psi$ be the projection $\Q[H] \to A$.
For all $h \in H$, the adjoint involution $\nu$ associated to the form satisfies
$\nu( \psi(h) \otimes 1) =  (\psi(h) \otimes 1)^{-1}$. Thus, $\nu$ preserves $A = A \otimes 1$, and 
by Lemma \ref{lemma:unique} $\nu|_A = \tau$. Since $\nu$ is $\C$-linear, $L \otimes 1 = 1 \otimes L$,
and $\nu|_{L \otimes 1} = \tau$, the involution $\tau$ is $L$-linear and thus of the first kind. Moreover, $\nu$ is the extension of $\tau$,
so they are of the same type and $\nu$ is of orthogonal type (resp. symplectic type) when the invariant bilinear form is symmetric
(resp. alternating).
\end{proof}

Before moving on to the next characterization, we first prove Lemma \ref{lemma:cyccenter} and a lemma about involutions when projecting $\R[H]$ to a simple $\R$-algebra.

\begin{proof}[Proof of Lemma \ref{lemma:cyccenter}]
First we show $L$ is a subfield of a cyclotomic field. The reduced trace (over $L$) is an $L$-linear map $\trd_{A/L}: A \to L$, and it is clear that
it is surjective. (Indeed, $\trd_{A/L}(L) = L$.) Choose some embedding $L \to \C$ and an isomorphism 
$\varphi: A \otimes_L \C \cong \Mat_n(\C)$, and let $r: \Q[H] \to A$ be the projection. Some power of the image $\varphi(r(h) \otimes 1)$ is
the identity, and so its trace (which is $\trd_{A/L}(r(h))$ by definition) is a sum of roots of unity. The composition $\trd_{A/L} \circ r$ is $\Q$-linear and $H$ spans the image $L$ as a
$\Q$-vector space. It is clear from this that $L$ is a subfield of a cyclotomic field. The second claim follows from the fact that subfields of cyclotomic fields
are all Galois extensions of $\Q$.
\end{proof}

\begin{lemma} \label{lemma:realinvols}
Suppose there is a surjection of $\R$-algebras $\psi: \R[H] \to B$ where $B$ is simple. There is an $\R$-linear involution $\nu$ on $B$ such that 
$\nu(\psi(h)) \psi(h) = \Id$ for every $h \in H$ and 
\begin{itemize}
\item $\nu$ is of the first kind and orthogonal type if and only if $B \cong \Mat_m(\R)$ for some $m \in \N$,
\item $\nu$ is of the first kind and symplectic type if and only if $B \cong \Mat_m(\bbH)$ for some $m \in \N$,
\item $\nu$ is of second kind if and only if $B \cong \Mat_m(\C)$ for some $m \in \N$, in which case
the center of $B$ is $\C$ and $\C^\nu = \R$.
\end{itemize}
\end{lemma}
\begin{proof}
First, note that since the only finite-dimensional division algebras over $\R$ are $\R, \C,$ and $\bbH$, it follows that 
$B$ is isomorphic to one of $\Mat_m(\R), \Mat_m(\C), \Mat_m(\bbH)$ for some $m \in \N$. Let respectively $D$ be one of $\R, \C,$ or $\bbH$.
We can without loss of generality identify $B$ with $\Mat_m(D)$, and thus obtain a corresponding action on $D^m$.

In the cases of $D = \R, \C$, we can average the standard
inner or Hermitian product on $D^m$ under $H$ to obtain an $H$-invariant inner or Hermitian product. In the case of $\bbH$,
there is also a standard Hermitian product, namely $v \cdot w = \sum_{i=1}^m v_i \overline{w_i}$ where
$\overline{w}$ is the involution on $\bbH$ given by $\overline{a + bi + cj + dk} = a - bi - cj - dk$. In this case also,
averaging the standard Hermitian product yields an $H$-invariant Hermitian product. In each case $v \cdot v > 0$ for
nonzero $v$ and the average has the same property. By a $D$-linear change of basis via Gram-Schmidt, 
we may assume the $H$-invariant inner or Hermitian products are the standard ones. We can take $\nu(b) = b^t$ in the case of $D = \R$ 
and $\nu(b) = \overline{b}^t$ in the cases of $D = \C, \bbH$. This involution is easily
checked to satisfy the claimed properties. Uniqueness of $\nu$ by Lemma \ref{lemma:unique} allows us to conclude
the ``only if'' in each of the three statements.
\end{proof}

\begin{proof}[Proof of Proposition \ref{prop:typefromrealrep}]
Let $K' = e^{-1}(\R)$. Implicitly, when we tensor over
$K'$ or $L$, we will do it via the map $e$. If $K' \neq L$, then we have an isomorphism
$e': L \otimes_{K'} \R \to \C$ satisfying $e'(\ell \otimes \lambda) = e(\ell) \lambda$, and so
by associativity of tensor products, there are isomorphisms
$$A \otimes_{K'} \R \cong A \otimes_L (L \otimes_{K'} \R) \cong A \otimes_L \C.$$
In either case ($L = K'$ or $L \neq K'$), $A \otimes_{K'} \R$ is isomorphic to $A$ tensored 
over its center $L$ with a simple $\R$-algebra ($\R$ or $\C$) and hence is simple. 
We have $A \otimes_{K'} \R \cong \Mat_n(\R), \Mat_n(\C),$ or $\Mat_{n/2}(\bbH)$.

Let $\varphi: \Q[H] \to A$ be the projection. There is an induced homomorphism 
$\R[H] = \Q[H] \otimes_\Q \R \to A \otimes_{K'} \R$, and since
the image of $H$ generates $A$ over $\Q$ (and $K'$), this map is surjective. Let $\nu$ be the involution
on $A \otimes_{K'} \R$ as in Lemma \ref{lemma:realinvols}, which is unique by Lemma \ref{lemma:unique}.
Since $\nu(\varphi(h)) = \varphi(h)^{-1}$ and $\varphi(h) \in A \otimes 1$, the involution $\nu$
restricts to an involution on $A = A \otimes 1$ which is $K'$-linear (hence $\Q$-linear). By uniqueness
(Lemma \ref{lemma:unique}), $\nu|_A = \tau$.
By Lemma \ref{lemma:realinvols}, the $\nu$-fixed subfield of the center in all cases is $\R$ whose intersection with
$e(L)$ is $e(K')$, and agreement of $\nu$ and $\tau$ on $A$ implies $K' = K$.

The above argument implies that if $L = K$, then all embeddings of $L$ in $\C$ are real-valued;
consequently, $A \otimes_K \R$ has an involution of the first kind and orthogonal or symplectic type
and so $A \otimes_K \R$ is isomorphic to either $\Mat_n(\R)$ or $\Mat_{n/2}(\bbH)$ respectively. Conversely,
if $A \otimes_K \R$ is either $\Mat_n(\R)$ or $\Mat_{n/2}(\bbH)$, then $\nu$ fixes the center and so fixes
$L$ which implies $\tau$ is of the first kind. The remaining cases are $L \neq K$ and 
$A \otimes_K \R \cong \Mat_n(\C)$ and so they coincide.
\end{proof}

To prove Proposition \ref{prop:typefromchi}, we first establish some facts for $\iota \chi$ when $\chi$ is an irreducible $\C$-character. We recall some theorems relating
$\iota \chi$ to the existence of certain invariant bilinear forms on irreducible representations and translate them into facts about
involutions. We will then bootstrap the results to $\Q$ to prove the proposition. We recall Theorem 23.16 from \cite{JameLieb}
with some mild modification. Note that for irreducible characters (over $\C$), $\iota\chi$ takes only values in $\{-1,0,1\}$.

\begin{theorem} \label{theorem:invtform}
Suppose $V$ is an irreducible $\C[H]$-module with character $\chi$. Then $\iota\chi \neq 0$ if and only if $V$ admits
a non-zero $H$-invariant bilinear form, and in this case, the form is nondegenerate. Furthermore, the form is symmetric
if and only if $\iota\chi =1$ and skew-symmetric if and only if $\iota\chi = -1$. 
\end{theorem}
\begin{proof}
This is precisely Theorem 23.16 from \cite{JameLieb} except for our claim that the forms are nondegenerate. (In \cite{JameLieb}, it only says ``non-zero''.)
The form must be non-degenerate by Proposition \ref{prop:typefrominvtform}.
\end{proof}

By Proposition \ref{prop:typefrominvtform}, proving Proposition \ref{prop:typefromchi} requires only that we translate information about
complex traces to rational ones. Recall that $A$ is a simple factor in the decomposition of $\Q[H]$. Let $W$ be the corresponding 
irreducible $\Q$-representation. Let $\mc A = A \otimes_L \C$. We have a surjection $\C[H] \cong \Q[H] \otimes_\Q \C \to \mc A \cong \Mat_n(\C)$
and a corresponding irreducible $\C$-representation $V$. Unfortunately, it is not always the case that $V$ is $W$ tensored with $\C$,
so we must expend some effort to relate the characters.

First, we recall the definitions of various traces for algebras over fields and some well-known facts relating them.
Suppose $A$ is a finite-dimensional $\Q$-algebra with 
center $L$. Then, $A$ acts on itself by left multiplication and this gives a representation 
$\eta_\ell: A \to \End_L(A)$. Define the trace of $a \in A$ over $L$ to be
$$\Tr_{A/L}(a) = \tr(\eta_\ell(a)).$$
Furthermore, for any field $k$ contained in $L$, left multiplication by $a$ is $k$-linear 
and we can consider the trace as a $k$-linear map on the $k$-vector space $A$. Denote this new trace
$\Tr_{A/k}(a)$.

If $k$ is a subfield of $L$, then as for $A$, we can view elements in $L$ as $k$-linear maps on
$L$ by left multiplication. As before we have $\Tr_{L/k}(\alpha)$ for $\alpha \in L$. 
We recall the following elementary fact about traces \cite[Section 9]{Rein}.
In particular, for the algebra $A$, we have $\Tr_{A/k}(a) = \Tr_{L/k} (\Tr_{A/L}(a))$. 
\begin{lemma}
Suppose $V$ is an $L$-vector space and $T: V \to V$ is $L$-linear. Let $T_k$
be $T$ viewed as a $k$-linear map on the $k$-vector space $V$. Then,
$$\tr(T_k) = \Tr_{L/k}(\tr(T).)$$
\end{lemma}
We relate the trace $\Tr_{A/L}$ and the corresponding character for irreducible
representations.

\begin{lemma} \label{lemma:tracetochar}
Suppose $V$ is $k$-vector space and $H \to \GL(V)$ is an irreducible $k$-representation and suppose $A$ is the corresponding simple $k$-algebra
in the decomposition of $k[H]$. Let $\chi$ be the character of $V$. Then, for $a \in A$,
$$ m \chi(a) = \Tr_{A/k}(a)$$
where $A \cong V^m$ as $k[H]$-modules.
\end{lemma}
\begin{proof}
By definition of character, $\chi(a)$ is the trace of $a$ via the action of $k[H]$, and hence $A$, on $V$.
The action of $A$ on itself block-diagonalizes as $m$ copies of its action on $V$ and the result easily follows.
\end{proof}

Recall the reduced trace $\trd_{A/L}$ for a central simple $L$-algebra $A$ from Section \ref{subsection:gofinindex}.
In addition to the properties discussed in that section, we have that $\Tr_{A/L} = n \trd_{A/L}$ where $n^2 = \dim_L(A)$.
(See \cite[Section 9]{Rein}.)
Now, we have the tools to prove Proposition \ref{prop:typefromchi}.
\begin{proof}[Proof of Proposition \ref{prop:typefromchi}]
As above, set $\mc A = A \otimes_L \C$ which has an isomorphism $\varphi: \mc A \to \Mat_n(\C)$, and let $V$
be the corresponding irreducible $\C$-representation of $H$. Let $\psi$ be
the corresponding character. Proposition \ref{prop:typefrominvtform} and Theorem \ref{theorem:invtform} 
imply that it is sufficient to show that
$\iota \psi$ and $\iota\chi$ are positive multiples of each other.

Now suppose $\eta_A: H \to A^\times$ is the representation from $H$ to the invertible elements
of $A$ and $\eta_{\mc A}: H \to \mc A^\times$ is the composition of $\eta_A$ with the canonical $A \to A \otimes_L \C= \mc A$.
Note that for all $h \in H$, by definition, $\psi(h)$ is the trace of $\eta_{\mc A}(h)$ which is the reduced
trace (over $L$) of $\eta_A(h)$. Thus,
$$|H| \iota \psi = \sum_{h \in H} \psi(h^2) = \sum_{h \in H} \trd_{A/L}(\eta_A(h^2)).$$
By the properties of reduced trace,
$$\sum_{h \in H} \trd_{A/L}(\eta_A(h^2)) = \frac{1}{n} \sum_{h \in H} \Tr_{A/L}(\eta_A(h^2)). $$
Now, suppose $m$ is such that $A \cong V^m$ where $V$ is the irreducible
representation corresponding to $A$. Since $\iota \psi \in \{-1, 0, 1\}$, we further have that if $d = \dim_\Q(L)$, 
$$|H| \iota \psi =\frac{|H|}{d}\Tr_{L/\Q}(\iota \psi) =  \frac{1}{nd} \Tr_{L/\Q} (\sum_{h \in H} \Tr_{A/L}(\eta_A(h^2))) $$
$$ =  \frac{1}{nd} \sum_{h \in H} \Tr_{A/\Q}(\eta_A(h^2))
		=  \frac{m}{nd} \sum_{h \in H} \chi(h^2) = \frac{|H|m}{nd} \iota \chi.$$
where the fourth equality follows from Lemma \ref{lemma:tracetochar}.
\end{proof}

We now prove Theorem \ref{theorem:Gtype}. As we have been doing, we drop the subscript $i$ in our proof.
\begin{proof}[Proof of Theorem \ref{theorem:Gtype}]
Recall that $\mc G = \Aut_A(A^{2g-2}, \langle-,-\rangle)$, and that $\sigma$ is the adjoint involution defined on $B=\Mat_{2g-2}(A^{op})$
and associated to $\langle-,-\rangle$.
Extend $\sigma$ (uniquely) to a $\C$-linear involution $\tilde{\sigma}$ on $\displaystyle B \underset{K}{\otimes} \C$
and let $$G = \{ C \in B \underset{K}{\otimes} \C \; | \;  \tilde \sigma(C) C = I \}.$$
Then, $G$ is a $K$-defined algebraic group whose $K$-points are $\mc G$.
Let $n^2 = \dim_L(A)$.

First, suppose that $\tau$ is of the first kind.
By Lemma \ref{lemma:changeinv}, there is an isomorphism $\varphi: B \otimes_K \C \to \Mat_{(2g-2)n}(\C)$
such that $\nu = \varphi \tilde \sigma \varphi^{-1}$ is either involution (1) or (2).
The involution $\nu$ has the same type as $\tilde \sigma$ and $\sigma$,
which, by Lemma \ref{lemma:oppositetype}, have the opposite type of $\tau$. Thus $\varphi$ maps $G$ isomorphically onto 
$\Sp_{(2g-2)n}(\C)$ if $\tau$ is of orthogonal type and onto $\OO_{(2g-2)n}(\C)$ otherwise.

If $\tau$ is of the second kind, then there is an isomorphism $\varphi: B \otimes_K \C \to \Mat_{(2g-2)n}(\C) \times \Mat_{(2g-2)n}(\C)$
such that $\nu = \varphi \tilde \sigma \varphi^{-1}$ is the involution (3). Arguments similar to the above show $G \cong \GL_{(2g-2)n}(\C)$.
The remaining equivalences follow from Propositions \ref{prop:typefrominvtform} and \ref{prop:typefromrealrep}.
\end{proof}

\section{Submodules of $M_i$}
\label{section:isotropic}
%!TEX root = ./MCGReps.tex
In this section we analyze further the structure of $\hat{R}$, when $p$ is $\phi$-redundant. In this case,
we show that $\hat R$ decomposes into two isomorphic totally isotropic submodules. Later, in Section \ref{section:liftrep},
we will see that the image of the handlebody group under $\rho$ preserves one of the submodules and thus
has a block uppertriangular form.
Moreover,  we will show that
when $p$ is $\phi$-redundant, there is an explicit rank two $\Q[H]$-submodule of $\hat{R}$ which allows us (in Section 
\ref{section:liftrep}) to produce 
two tightly controlled unipotent elements in the image of $\rho_{H,p}$.

Recall that if $\phi: T_g \to F_g$ is a surjective homomorphism, we say that an epimorphism $p: T_g \to H$
is $\phi$-redundant if $p$ factors through an epimorphism $p' : F_g \to H$ which is redundant (i.e. the kernel
contains a free generator). As mentioned in the introduction,
every epimorphism $\phi: T_g \to F_g$ arises as follows \cite[Theorem 5.2]{Jaco}.
Let $\mc H = \mc H_g$ be a genus $g$ handlebody and pick some identification of $\partial \mc H$ with
$\Sigma_g$. The inclusion $\Sigma_g \hookrightarrow \mc H$ induces a map on the fundamental
groups which is surjective. Since the fundamental group of $\mc H$ is a free group, this induces a 
surjective map $\phi: T_g \to F_g = \pi_1(\mc H)$ where $F_g$ is the free group of rank $g$. Henceforth, we will simply
refer to $\pi_1(\mc H)$ as $F_g$.

Now fix some $\phi: T_g \to F_g$ arising from an identification $\Sigma \cong \partial \mc H$, fix some $\phi$-redundant $p$, and set $S = \ker{p'}$ 
for the corresponding $p': F_g \to H$. Let $\tilde \Sigma \to \Sigma$ be the cover corresponding to $p$.

\subsection{The handlebody group and $\Aut(F_g)$}
In this section, we recall a basic fact about handlebody groups. Let $\ast$ be some point on the boundary of $\mc H$. The
handlebody group with fixed point, $\Map(\mc H, \ast)$, is the subgroup of those mapping classes in $\Mod(\Sigma, \ast)$ which
contain a representative homeomorphism extending to a homeomorphism of $\mc H$.
Recall that the handlebody group (without fixed point) $\Map(\mc H)$ is the similarly defined subgroup of $\Mod(\Sigma)$.

\begin{theorem}{\cite{Zies, Grif, McMi} } \label{theorem:handlebodytoautFg}
The natural homomorphisms $\Map(\mc H, \ast) \to \Aut(F_g)$ and $\Map(\mc H) \to \Out(F_g)$ are surjective.
\end{theorem}

\subsection{Free submodule of rank two}
We start by finding a special torus with one boundary component embedded in $\Sigma$.
Let $\ast \in \Sigma = \partial \mc H$ be a basepoint for the 
fundamental groups of both $\mc H$ and $\Sigma$.
\begin{lemma} \label{lemma:specialtorus}
Let $p:T_g \to H$ be $\phi$-redundant and $p': F_g \to H$ the induced map. Then, $\Sigma$ contains a subsurface $\Sigma'$ homeomorphic
to a torus with one boundary component such that
the image of $\pi_1(\Sigma', \ast)$ in $\pi_1(\Sigma, \ast) = T_g$ lies in the kernel of $p$. Moreover, there are simple closed curves $a, b$ in 
$\Sigma'$ such that
\begin{itemize}
\item $a$ and $b$ intersect once, and $\langle a, b \rangle_{\Sp} = 1$
\item the homology classes of $a, b$ generate $\HH_1(\Sigma', \Q)$
\item $b$ bounds a disc in $\mc H$
\item $a, b$ avoid $\ast$.
\end{itemize}
\end{lemma}
\begin{proof}
Let $\alpha$ be a free basis element of $F_g$ lying in the kernel of $p'$. This exists by the assumption that $p$ is $\phi$-redundant. 
We first show that the homotopy class $\alpha$ contains some simple closed curve $a$ supported on $\Sigma$.
Moreover, we show that there is an oriented simple closed curve $b$ on $\Sigma$ passing through $\ast$ such that $b$ intersects $a$ transversely at one point
and the based homotopy class of $b$ lies in the kernel of $\phi$. (Note that since we define $\Mod(\Sigma)$ as $\Homeo^+(\Sigma)/\Homeo_0(\Sigma)$, we
are implicitly working in the category of topological spaces and so ``transverse intersection'' strictly speaking has no meaning. However, $\Mod(\Sigma)$ is
also $\Diff^+(\Sigma)/\Diff_0(\Sigma)$, and we can just as well work in the smooth category.)
To see all this, note that by assumption there is some free basis $\alpha = \alpha_1, \dots, \alpha_g$ of $F_g$.
Recall that $\phi: T_g \to F_g$ is the natural map on fundamental groups. By explicit construction, one can find a (potentially different)
free basis $\alpha_1', \dots, \alpha_g'$ of $F_g$ such that each homotopy class $\alpha_i'$ contains a simple closed curve $a_i'$ on $\Sigma$;
moreover, we can ensure there is some oriented simple closed curve $b'$ passing through $\ast$ such that $a_1'$ and $b$ intersect once transversely and $b'$ bounds a disc
in $\mc H$ (so $b'$ lies in the kernel of $\phi$). See Figure \ref{figure:fgbasis}.

\begin{figure}[htbp] \centering
	\includegraphics[width=.8\textwidth]{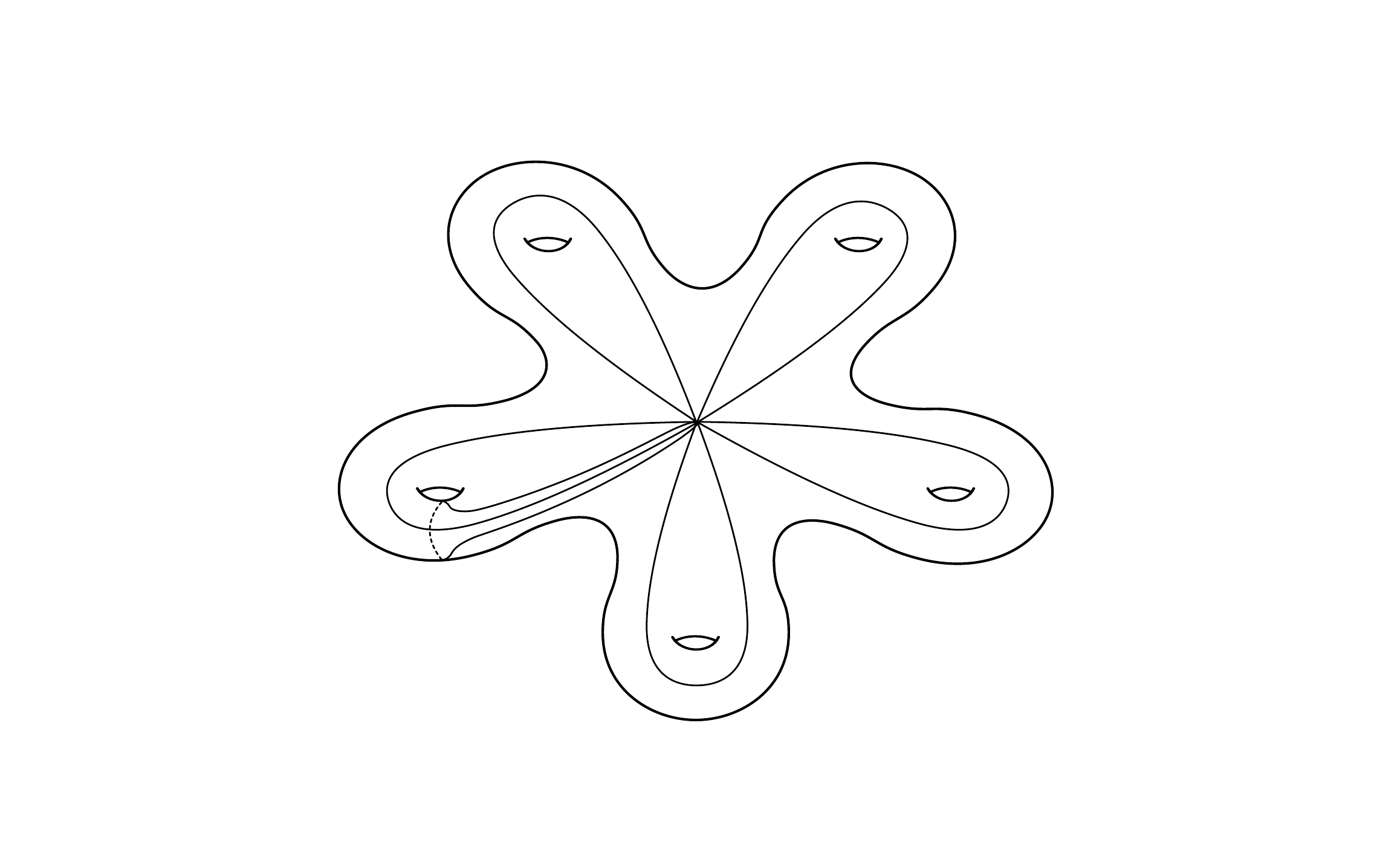}
	\caption{The curves $a_1', \dots, a_g'$ and $b'$ in the proof of Lemma \ref{lemma:specialtorus} for $g = 5$.}
	\label{figure:fgbasis}
\end{figure}

Since $\Map(\mc H, \ast) \to \Aut(F_g)$ is surjective (Theorem \ref{theorem:handlebodytoautFg}),
there is some homeomorphism $\varphi$ of the handlebody fixing $\ast$ and mapping $\alpha_1'$ to $\alpha$. Let $a = \varphi(a')$ and
$b = \varphi(b')$. Note that $b$ bounds a disc in $\mc H$ as well.

Since $a$ and $b$ intersect once transversely, a regular neighborhood $\Sigma'$
of $a$ and $b$ in $\Sigma$ is a torus with one boundary component and with fundamental group generated by the homotopy classes of $a$ and $b$.
Since (the homotopy classes of) $a$ and $b$ both lie in the kernel of $p$, it follows that (the image of) $\pi_1(\Sigma', \ast)$ lies in the kernel.

As it stands, the curves $a, b$ satisfy the first three properties required by the lemma
but $a$ and $b$ do not avoid $\ast$. This is easily fixed by an isotopy of the curves.
\end{proof}

Let $\Sigma', a, b$ be as in the lemma. Then, the preimage of $\Sigma'$ under the cover
$\tilde \Sigma \to \Sigma$ consists of $|H|$ disjoint homeomorphic copies of $\Sigma'$. Label these surfaces as $\tilde \Sigma_h'$ as
$h$ ranges over $H$ so that $d_h(\tilde \Sigma_{1_H}') = \tilde \Sigma_h'$ where $d_h$ is the deck transformation induced by $h \in H$.

Similar to $\Sigma'$, the preimages of $a$ and $b$ each consist of $|H|$ simple closed curves which project homeomorphically
to $a$ and $b$ via the covering map. Call the preimages $\tilde a_h$ and $\tilde b_h$ as $h$ ranges over $H$ so 
that $a_h, b_h$ lie in $\tilde \Sigma_h'$. Then, $a_h = d_h(a_{1_H})$ and $b_h = d_h(b_{1_H})$.
For a simple closed curve $c$, we denote its homology class by $[c]$.

\begin{lemma} \label{lemma:rank2submodule}
Let notation be as above. Then $[\tilde a_{1_H}]$ (similarly $[\tilde b_{1_H}]$)
generates a submodule of $\HH_1(\tilde \Sigma, \Q)$ isomorphic to $\Q[H]$. Together, $[\tilde a_{1_H}]$ and $[\tilde b_{1_H}]$ generate a submodule
isomorphic to $\Q[H]^2$. Moreover, $\langle [\tilde a_{1_H}], [\tilde b_{1_H}] \rangle = 1$.
\end{lemma}
\begin{proof}
Consider the disjoint union $\tilde  \Sigma' = \cup_{h \in H} \Sigma_h'$.
From the Mayer-Vietoris sequence, one can deduce that $\HH_1(\tilde \Sigma', \Q)$
is a direct summand as a $\Q$-vector space of $\HH_1(\tilde \Sigma, \Q)$.
Since the action of $H$ preserves $\tilde \Sigma'$ and its complement,
$\HH_1(\tilde \Sigma', \Q)$ is a summand as a $\Q[H]$-module.
As $\tilde \Sigma'$ is the disjoint union of the $\tilde \Sigma_h'$, it follows that
$\HH_1(\tilde \Sigma', \Q) \cong \bigoplus_{h \in H} \HH_1(\tilde \Sigma_h', \Q)$.
Each $ \HH_1(\tilde \Sigma_h', \Q)$ is $2$-dimensional so $\dim_{\Q} \HH_1(\tilde \Sigma', \Q) = 2|H|$.

Now, we know that for each $h \in H$, the classes $[\tilde a_h]$ and $[\tilde b_h]$ generate $ \HH_1(\tilde \Sigma_h', \Q)$, and so
$[\tilde a_{1_H}], [\tilde b_{1_H}]$ generate $\HH_1(\tilde \Sigma', \Q) $ as a $\Q[H]$-module. Because of dimension, we must have
$[\tilde a_{1_H}], [\tilde b_{1_H}]$ freely generate, and  $\HH_1(\tilde \Sigma', \Q) \cong \Q[H]^2$. It follows that each of $[\tilde a_{1_H}]$ and $[\tilde b_{1_H}]$
generate a module isomorphic to $\Q[H]$.
For the last claim, recall that by definition, 
$$\langle [\tilde a_{1_H}], [\tilde b_{1_H}] \rangle = \sum_{h \in H} \langle [\tilde a_{1_H}], h \cdot [\tilde b_{1_H}] \rangle_{\Sp} h$$
Since $h \cdot [\tilde b_{1_H}] = [\tilde b_h]$ is disjoint from $\tilde a_{1_H}$, the only term which survives is that for $h = 1_H$.
Since $\langle [\tilde a_{1_H}],   [\tilde b_{1_H}] \rangle_{\Sp} = \langle [a],   [b] \rangle_{\Sp} = 1$,
the claim follows.
\end{proof}

\subsection{Totally isotropic subspaces}

The inclusion $\Sigma_g \hookrightarrow \mc H$ induces a surjective map $\HH_1(\Sigma_g, \Q) \to \HH_1(\mc H, \Q)$.
It follows that the action of a homeomorphism of the handlebody preserves the kernel.
It is easy to see that the kernel of this map is {\it totally isotropic} relative to the alternating intersection form $\langle-,-\rangle_{\Sp}$.
I.e. for any two elements $\beta, \beta'$ in the kernel, $\langle \beta, \beta' \rangle_{\Sp} = 0$.  Moreover, this subspace has a complementary
totally isotropic subspace which projects onto $\HH_1(\mc H, \Q)$. We want to show similarly that $\hat{R} = \HH_1(\tilde \Sigma, \Q)$ has 
an isotropic submodule and decomposition.

Now, let $\hat{S} = (S/[S, S]) \otimes_{\Z} \Q$ which is naturally identified with $\HH_1(\tilde {\mc H}, \Q)$, the
rational first homology of the cover $\tilde{\mc H} \to \mc H$ corresponding to the inclusion $S \to F_g$. Let
$\hat{P}$ be the kernel of the map $\hat{\phi}_R :\hat{R} \to \hat{S}$ induced by $\phi$.
Note that this kernel is equivalent to the kernel of the map $\HH_1(\tilde \Sigma, \Q) \to \HH_1(\tilde{\mc H}, \Q)$.
\begin{lemma} \label{lemma:H1decomp}
The subspace $\hat P$ is a $\Q[H]$-submodule which is totally isotropic relative to the sesquilinear form $\langle-,-\rangle$.
Moreover, as $\Q[H]$-modules, $\hat{R} \cong \hat{P} \oplus \hat{S}$ and 
both $\hat{P}$ and $\hat{S}$ are isomorphic to $\Q[H]^{g-1} \oplus \Q$.
\end{lemma}
We remark that below
we show that the isomorphism $\hat{P} \oplus \hat{S} \cong \hat{R}$ can be chosen such that the image of $\hat{S}$
is totally isotropic relative to $\langle -,-\rangle$.
\begin{proof}
As mentioned above, the map $\HH_1(\Sigma, \Q) \to \HH_1(\mc H, \Q)$ has totally isotropic kernel relative
to the alternating intersection form. Applying the same fact 
to the map on the covers $\HH_1(\tilde \Sigma, \Q) \to \HH_1(\tilde{\mc H}, \Q)$ implies
that $\hat P$ is totally isotropic relative to $\langle -,-\rangle_{\Sp}$. Since the sesquilinear form is defined as a sum with
coefficients in terms of $\langle -, - \rangle_{\Sp}$, it follows that $\hat P$ is also totally isotropic relative to the sesquilinear form.

It follows from the definitions that $\hat{\phi}_R$ is a $\Q[H]$-homomorphism, so its kernel, $\hat{P}$, is a $\Q[H]$-submodule.
By Theorem \ref{theorem:gaschutz}, $\hat{S} \cong \Q[H]^{g-1} \oplus \Q$. This fact combined with Proposition
\ref{prop:surfacegaschutz} and surjectivity of $\hat{\phi}_R$ implies the rest of the Lemma since $\Q[H]$ is semisimple.
\end{proof}

Now consider the above situation projected to the $i$th isotypic component $M_i$ of $\hat{R}$ using the
notation of the previous sections. The above decomposition
projects to a decomposition of $M_i$, and we can conclude a stronger statement relative to the sesquilinear form. Let
$M_i'$ be the projection of $\hat{P}$ to $M_i$.

\begin{lemma} \label{lemma:niceAibasis}
Let $\tilde a_{1_H}, \tilde b_{1_H}$ be as in Lemma \ref{lemma:rank2submodule} and $\tilde \alpha_i, \tilde \beta_i$ the projection of their
homology classes to $M_i$. The module $M_i'$ is totally isotropic relative to $\langle-,-\rangle$, and there is a subspace $M_i''$ of $M_i$ such that
\begin{itemize}
\item $M_i''$ is totally isotropic relative to $\langle-,-\rangle$,
\item $M_i = M_i' \oplus M_i''$
\item there are free $A_i$-bases $\tilde \beta_i = m_{i,1}', m_{i,2}', \dots, m_{i, g-1}'$
(resp. $\tilde \alpha_i = m_{i,1}'', m_{i,2}'', \dots, m_{i, g-1}''$) of $M_i'$ (resp. $M_i''$) such that $\langle m_{i,j}'', m_{i,k}' \rangle = \delta_{j,k}$
\end{itemize}
\end{lemma}
\begin{proof}
The module $M_i'$ is totally isotropic since $\hat P$ is.
Notice that by the choice of $\tilde b_{1_H}$, its homology class lies in $\hat{P}$. The homology class of $\tilde b_{1_H}$ generates a 
copy of $\Q[H]$ in $\hat{P}$, and so $[\tilde b_{1_H}]$ necessarily
generates a copy of $A_i$ in $M_i$. Thus, $\tilde \beta_i$ can be extended to some basis $\tilde \beta_i = m_{i,1}', m_{i,2}', \dots, m_{i, g-1}'$
of $M_i'$. Furthermore, we can ensure $\langle \tilde \alpha_i, m_{i,j}' \rangle = \delta_{1j}$. (If it is not already true, we can alter $m_{i,j}'$ for $j > 1$ by 
adding an appropriate multiple of $\tilde \beta_i$.)

Lemma \ref{lemma:H1decomp} implies that $M_i'$ is isomorphic to $A_i^{g-1}$, and so we immediately have some complementary
subspace $N_i$ which is isomorphic to $A_i^{g-1}$ (but not necessarily isotropic). Moreover, since $\tilde \alpha_i$ does not lie in $M_i'$, we can further arrange
that $N_i$ contains $\tilde \alpha_i$.
Consider the homomorphism to the dual space of $N_i$
\begin{align} M_i' &\to  \Hom_{A_i} ( N_i, A_i)  \nonumber \\ m_i' & \mapsto  \langle -, m_i' \rangle  \label{eqn:dual} \end{align}
Note that this is an $A_i$-module homomorphism where the (left) $A_i$-module structure on $\Hom_{A_i} ( N_i, A_i)$ is given
by $(a \cdot f)(n) = (f (n) ) \tau(a)$ for all $n \in N_i$. Since $M_i'$ is totally isotropic and $\langle-,-\rangle$ is nondegenerate, each $m' \in M_i'$
must pair nontrivially with some element in $N_i$. Consequently, the map in (\ref{eqn:dual}) has trivial kernel, and the fact that
$\Hom_{A_i} ( N_i, A_i) \cong A_i^{g-1}$ implies (\ref{eqn:dual}) is an isomorphism.

There is therefore a unique basis
$n_{i,1}, \dots, n_{i, g-1}$ of $N_i$ such that $\langle n_{i,j}, m_{i,k}' \rangle = \delta_{j,k}$.
Moreover, since $\tilde \alpha_i$ lies in $N_i$ and 
$\langle \alpha_i, m_{i,j}' \rangle = \delta_{1j}$,
it follows that $n_{i, 1} =  \tilde \alpha_i$. From $N_i$, we can construct a submodule $M_i''$ which has the
same properties as $N_i$ but is also isotropic.

Let $c_{j,k} = \langle n_{i,j}, n_{i,k} \rangle$ for $j \neq k$, and let $c_{j,j} = \frac{1}{2} \langle n_{i,j}, n_{i,j} \rangle$.
Note that since $\langle-,-\rangle$ is skew-Hermitian, $\tau(c_{j,k}) = - c_{k,j}$.
Let $m_{i,j}'' = n_{i,j} + \sum_{\ell \leq j} c_{j,\ell} m_{i,\ell}'$ and $M_i''$ be the submodule generated by $m_{i,1}'', \dots, m_{i,g-1}''$.
Note that $m_{i,1}'' = n_{i,1} =  \tilde \alpha_i$. It follows from the properties of the $n_{i,j}$ and the fact that $M_i'$ is isotropic
that $\langle m_{i,j}'', m_{i,k}' \rangle = \delta_{jk}$. Moreover, one computes that for all pairs $j > k$:
$$\begin{array}{rcl} \langle m_{i,j}'', m_{i,k}'' \rangle & = & \langle n_{i,j}, n_{i,k} \rangle + \langle \sum_{\ell \leq j}  c_{j,\ell} m_{i,\ell}', n_{i,k} \rangle
													+ \langle n_{i,j}, \sum_{\ell \leq k} c_{k,\ell} m_{i,\ell}' \rangle \\
									     & = &  c_{j,k} + \langle c_{j, k} m_{i,k}', n_{i,k} \rangle =  c_{j,k} - c_{j,k} = 0 . \end{array} $$
Similarly, 
$$\begin{array}{rcl} \langle m_{i,j}'', m_{i,j}'' \rangle & = & \langle n_{i,j}, n_{i,j} \rangle + 
												\langle c_{j,j} m_{i,j}', n_{i,j} \rangle + \langle n_{i,j}, c_{j,j} m_{i,j}' \rangle \\
									 &     = & 2 c_{j,j} - c_{j,j} + \tau(c_{j,j}) = 0 . \end{array}$$
Thus, $M_i''$ has all the desired properties.
\end{proof}

\begin{corollary}
The submodule $\hat P$ has a complementary totally isotropic isomorphic submodule in $\hat R$.
\end{corollary}
\begin{proof}
By Lemma \ref{lemma:niceAibasis}, there is a direct sum decomposition of $\hat R$ as a $\Q[H]$-module:
$$\hat R = \Q^{2g} \oplus \bigoplus_{i=1}^\ell M_i = \Q^{2g} \oplus \bigoplus_{i=1}^\ell (M_i' \oplus M_i'')$$
where $M_i'$ and $M_i''$ are isotropic and $M_i' = \hat P \cap M_i$. Similarly, $\Q^{2g} \cap \hat P \cong \Q^g$ and
there is a complimentary isotropic subspace $B$ isomorphic to $\Q^g$. The submodule complementary to $\hat P$ is
$B \oplus (\bigoplus_{i=1}^\ell M_i'')$.
\end{proof}

\section{The representation $\rho$ and the image of the handlebody group}
\label{section:liftrep}
%!TEX root = ./MCGReps.tex
We now provide more details about the representations $\rho_{H,p}$. As mentioned in the introduction, we will work with
$\Aut(T_g)^+ = \Mod(\Sigma_g, \ast)$ instead of $\Mod(\Sigma_g)$, and at the end (Section \ref{section:forgetmarkedpt}) we will use (virtual) arithmetic
quotients of the former to obtain (virtual) arithmetic quotients of the latter. As in the previous section,
let $p:T_g \to H$ be a $\phi$-redundant surjective homomorphism, and let $S$ be the kernel
of the induced map $p': F_g \to H$. We have the following commutative diagram
and short exact sequences.
$$ \begin{CD}  1 @>>> R @>>> T_g @>p>> H @>>>1 \\
		       @.         @VVV     @VV\phi V             @|         @. \\
		       1  @>>> S @>>> F_g @>p'>> H @>>> 1 \end{CD}    
$$

We set 
$$\Gamma_{H,p} = \{f \in \Aut(T_g)^+ \; | \; p \circ f = p \}.$$
Note that if $f \in \Gamma_{H,p}$, then it follows automatically that $f(R) = R$.
For such $f$, we can restrict the automorphism to $R$ and then project to an automorphism
of the abelianization $\overline{R}$. This, in turn, induces an automorphism of $\hat R$.
Denote this map as $\rho_{H, p}: \Gamma_{H,p} \to \Aut(\hat R)$.
For the remainder of this section, we will suppress the subscripts $H$ and $p$,  
so henceforth we fix some $p$ and $H$, and set $\rho = \rho_{H,p}$
and $\Gamma = \Gamma_{H,p}$.

As we have seen, $\hat R$ has a $\Q[H]$-module structure, and this is preserved by such $f$.
\begin{lemma} \label{lemma:rhotomoduleaut}
The image $\rho(\Gamma)$ lies in $\Aut_{\Q[H]}( \hat R)$.
\end{lemma}
\begin{proof}
Let $f \in \Gamma$, and let $\overline{f}$ be the induced action of $f$ on $\overline{R}$.
It suffices to check that $\overline{f}$ commutes with the action of $H$. Suppose $r \in R$ and $\overline{r} \in \overline{R}$
is its image in the quotient. For any $h \in H$, the action on $\overline{f}(\overline{r})= \overline{f(r)}$ is 
$h \cdot \overline{f}(\overline{r}) = \overline{h f(r) h^{-1}}$.
Since $f \in \Gamma$, it follows that $f$ is a trivial automorphism mod $R$, and so $r_0 f(h)  = h$ for some $r_0 \in R$ and 
$$\overline{h f(r) h^{-1}} = \overline{r_0 f(h r h^{-1}) r_0^{-1}} = \overline{f(h r h^{-1})} = \overline{f}(h \cdot \overline{r}).$$
\end{proof}

To show that $\rho(f)$ additionally preserves the form $\langle-,-\rangle$, we appeal to a topological reformulation.
Interpreted in a topological setting, the homomorphism can be viewed as lifting the homeomorphisms to the cover and using the induced action
on the first homology group. Let us be more precise. Suppose $\tilde \ast$ is a point in the fiber over $\ast \in \Sigma$ for the cover $\tilde \Sigma \to \Sigma$.
For a homeomorphism $\varphi \in \Homeo(\Sigma, \ast)$, the lift of $\varphi$ is the homeomorphism $\tilde \varphi$ fixing $\tilde \ast$ such that 
$\Pi \circ \tilde \varphi = \varphi \circ \Pi$ where $\Pi: \tilde \Sigma \to \Sigma$ is the covering map.
The homeomorphism $\varphi$ will lift precisely when the induced action of $\varphi$ on $\pi_1(\Sigma, \ast)$ preserves 
$R = \pi_1(\tilde \Sigma, \tilde \ast)$, i.e. its mapping class is an element of $\Gamma$. Moreover, this lifting induces a well-defined
homomorphism $\Gamma \to \Mod(\tilde \Sigma, \tilde \ast)$. Recall that $\hat R$ is naturally identified with $\HH_1(\tilde \Sigma, \Q)$.
If $f \in \Gamma$ and $\tilde f \in \Mod(\tilde \Sigma, \tilde \ast)$ is its lift, then $\rho(f) = \tilde f_*$ where $\tilde f_*$ is the induced
action on $\HH_1(\tilde \Sigma, \Q)$. We can now establish the following.

\begin{lemma} \label{lemma:liftpreserveform}
The action of $\rho(f)$ preserves the sesquilinear form $\langle-,-\rangle$ for all $f \in \Gamma$.
\end{lemma}
\begin{proof}
By the above discussion, $\rho(f)$ is equivalent to $\tilde f_*$ which is the action induced by 
the orientation-preserving mapping class $\tilde f$. Since such a homomorphism 
preserves $\langle-,-\rangle_{\Sp}$, and, in addition, $\rho(f)$ acts by a $\Q[H]$-automorphism by Lemma
\ref{lemma:rhotomoduleaut}, it must preserve $\langle-,-\rangle$.
\end{proof}

Since the image acts by $\Q[H]$-automorphisms, it decomposes into actions on each isotypic component. Consequently,
we obtain representations 
$$\rho_i = \rho_{H,p,i}: \Gamma \to \Aut_{A_i}(M_i, \langle-,-\rangle)$$
by projecting the action.

\subsection{Previous results for automorphisms of a free group}
Here, we recall a theorem from \cite{GrunLubo} which describes the image of $\Aut(F_n)$ for a representation analogous to $\rho$.
We will use this result to show that the image of the handlebody group maps (virtually) onto the $\frak O$-points of the Levi factor of a
parabolic subgroup of $\Aut_{A_i}(A_i^{2g-2}, \langle-,-\rangle)$.
The essential connection between the two
is that $\Map(\mc H, \ast)$ surjects onto $\Aut(F_g)$ via the action on its fundamental group.

We set up some additional notation and define some new terms.  For any group $X$ and subgroup $Y$, we set 
$$\Aut(X | Y) = \{ f \in \Aut(X) \; | \; f(Y) = Y \}.$$
Moreover, for $Y$ normal, let $\Aut^t(X | Y)$ be the subgroup of $\Aut(X | Y)$ consisting of automorphisms acting trivially on $X/Y$.
Note that in this notation $\Gamma = \Aut(T_g)^+ \cap \Aut^t(T_g | R)$.
We have a homomorphism $\eta = \eta_{H,p} : \Aut^t(F_g | S) \to \Aut_{\Q[H]}(\hat{S})$
defined similarly to $\rho$; namely restrict the automorphism to $S$, project to the abelianization $\overline{S}$, and take the induced map
on $\hat S = \overline{S} \otimes_\Z \Q$.

By Gasch\"utz's theorem (Theorem \ref{theorem:gaschutz}), $\hat S \cong \Q[H]^{g-1} \oplus \Q$. Consequently, for each $i$, we
obtain, via projection, a representation 
$$\eta_{i} : \Aut^t(F_g | S) \to \Aut_{A_i}(A_i^{g-1}) = \GL_{g-1}(A_i^{op}).$$
Recall that $A_i$ is simple and isomorphic to $\Mat_{m_i}(D_i)$ for some division algebra $D_i$ and integer $m_i$.
Thus, we can furthermore identify the target space as follows.
$$\GL_{g-1}(A_i^{op}) \cong \GL_{g-1}(\Mat_{m_i}(D_i)^{op}) \cong \GL_{(g-1)m_i}(D_i^{op})$$
The special linear group $\SL_{(g-1)m_i}(D_i^{op})$ is the group of matrices $C$ satisfying $\nrd_{B/L}(C) = 1$
for $B = \End_{D_i}(D_i^{m_i(g-1)})$.

Now, suppose $K$ is some number field, and suppose $D$ is a finite-dimensional $K$-algebra. Then a subring
$\mc R \subseteq D$ is an {\it order} if it contains a $K$-basis of $D$ and it is a finitely generated $\Z$-module. Recall that $A_i$ is 
finite-dimensional over $K_i$, (the $\tau$-fixed subfield of its center), and so $D_i$ is a finite-dimensional $K_i$-algebra.
The next theorem is a direct consequence of Theorem 5.5 of \cite{GrunLubo}.

\begin{theorem}{(Grunewald--Lubotzky)} \label{theorem:GrunLubo}
Keeping notation as above, if $g \geq 4$, then there exists an order $\mc R_i$ of $D_i^{op}$ such that the intersection $\eta_i(\Aut^t(F_g | S)) \cap \SL_{m_i(g-1)}(\mathcal R_i)$
is of finite index in $\SL_{m_i(g-1)}(\mathcal R_i)$.
\end{theorem}

\begin{remarknum} \label{remark:stillzardense}
Theorem 5.5 in \cite{GrunLubo} is proved for $g \geq 4$. However, as explained on page 1592 there, it is also true when $g = 3$
in many cases, e.g. if $m_i > 1$. Moreover, even if $m_i = 1$, the image of $\eta_i$ contains finite index subgroups of the upper
and lower unitriangular subgroups of $\SL_{m_i(g-1)}(\mc R_i)$. The issue is that it is not clear there that these unitriangular subgroups generate
$\SL_{m_i(g-1)}(\mc R_i)$. But in any event, they generate a Zariski dense subgroup.
\end{remarknum}

\subsection{Structure of the image of the handlebody group}
Here, we will analyze the image of the handlebody subgroup $\Map(\mc H, \ast)$ of $\Mod(\Sigma, \ast)$. 
The handlebody group naturally has a homomorphism to $\Aut(F_g)$ via the action on the fundamental group
of $\mc H$.
We define the finite index subgroup
$$\Lambda = \Gamma \cap \Map(\mc H, \ast)$$

We will show that the image of $\Lambda$ under $\rho_{i}$ block uppertriangularizes, and furthermore
that the blocks on the diagonal range over a wide variety of matrices. More precisely, we have the following result.
Similar to Section \ref{section:symplecticgaschutz}, we set $C^* = (\tau(c_{ji}))$ for any matrix $C = (c_{ij}) \in \Mat_m(A_i^{op})$.
We define the parabolic subgroup of $\mc G$
$$\mc P = \left\{  \left( \begin{array}{cc} (C^*)^{-1} & E \\ 0 & C\end{array} \right) \; | \; C, E \in \Mat_{g-1}(A_i^{op}) \text{ and } E^* C = C^* E  \right \}.$$
Let $pr: \mc P \to \GL_{g-1}(A_i^{op})$ be the homomorphism 
$$\left( \begin{array}{cc} (C^*)^{-1} & E \\ 0 & C \end{array} \right) \mapsto C.$$
\begin{prop} \label{prop:handlebodyimage}
After identifying $\Aut_{A_i}(M_i)$ with $\GL_{2(g-1)}(A_i^{op})$ via 
the ordered basis $m_{i,1}', \dots, m_{i, g-1}', m_{i, 1}'', \dots, m_{i,g-1}''$ from Lemma \ref{lemma:niceAibasis},
we have
$$\rho_i(\Lambda) \subseteq \mc P$$
Moreover, if $g \geq 4$, then there is an order $\mc R_i$ of $D_i^{op}$ such that $pr(\rho_i(\Lambda))$
is a finite index subgroup in $\SL_{m_i(g-1)}(\mc R_i)$. If $g = 3$, then $pr(\rho_i(\Lambda))$ is Zariski
dense in $\SL_{m_i(g-1)}(\mc R_i)$.
\end{prop}

The first part of the proposition is essentially equivalent to the following lemma.

\begin{lemma} \label{lemma:invtsubspace}
The following is a well-defined commutative diagram.
$$\begin{CD}  \Lambda @>\rho>> \Aut_{\Q[H]}( \hat{R} | \hat{P})\\
                                  @VVV                                                                        @VVV               \\
		            \Aut^t(F_g | S)  @>\eta>>   \Aut_{\Q[H]}(\hat{S})              
\end{CD}$$
\end{lemma}
\begin{proof}
Note that $\Lambda$ can project to $\Aut^t(F_g | S)$ since $\Map(\mc H, \ast)$ 
preserves $\ker(\phi)$ and $\Gamma$ lies in $\Aut^t(T_g | R)$.
In addition to commutativity, we have to establish that $\rho(\Lambda)$ preserves the subspace $\hat P$.
We translate parts of the diagram into topological language where it becomes the following:
$$\begin{CD}\Map(\mc H, \ast) \cap \Gamma @>\rho>>    \Aut_{\Q[H]}( \HH_1(\tilde \Sigma, \Q) | \hat{P}) \\
                                  @VVV                                                                        @VVV               \\
		          \Aut^t(F_g | S) @>\eta>>   \Aut_{\Q[H]}(\HH_1(\tilde{ \mc H}, \Q))    
\end{CD}.$$

We prove commutativity in two steps. First, we can lift elements from $\Map(\mc H, \ast)$ to $\Map(\tilde{\mc H}, \tilde \ast)$.
and properties of lifts and the fact that $S = \pi_1(\tilde{ \mc H}, \tilde \ast)$ implies commutativity in the following diagram.
$$\begin{CD}   \Map(\mc H, \ast) \cap \Gamma @>\text{lift}>> \Map(\tilde{ \mc H}, \tilde \ast) \\
                                  @VVV                                                                        @VVV               \\
		    \Aut^t(F_g | S) @>>>   				\Aut(S)          
\end{CD}$$

Since elements of $\Map(\tilde{ \mc H}, \tilde \ast)$ preserve $\hat P$, the following diagram is well-defined and commutative.
$$\begin{CD}\Map(\tilde{ \mc H}, \tilde \ast)  @>>> \Aut( \HH_1(\tilde \Sigma, \Q) | \hat{P}) \\
                                  @VVV                                                                        @VVV               \\
		               \Aut(S) @>>>  		  \Aut(\HH_1(\tilde{ \mc H}, \Q))     
\end{CD}$$
Putting the diagrams together establishes the Lemma.
\end{proof}

This establishes the block diagonal form claimed in Proposition \ref{prop:handlebodyimage}.
Proving Proposition \ref{prop:handlebodyimage} is now a matter of combining some of the
above results.

\begin{proof}[Proof of Proposition \ref{prop:handlebodyimage}]
Let $\gamma \in \Lambda$.
Lemma \ref{lemma:invtsubspace} implies that $\rho_i(\gamma)$ preserves $M_i'$ and so $\rho_i(\gamma)$ block diagonalizes.
To show that $\rho_i(\gamma)$ lies in $\mc P$ as defined, it suffices to show that any element of $\mc G$ preserving $M_i'$
has diagonal entries $ (C^*)^{-1}, C$ for some matrix $C \in \Mat_{g-1}(A^{op})$.
Notice that relative to the basis, (and multiplying in $\Mat(-^{op})$)
$$\langle (a_1, \dots, a_{2(g-1)}), (b_1 , \dots, b_{2(g-1)}) \rangle = 
			(\tau(b_1) \dots \tau(b_{2(g-1)}) ) \left( \begin{array}{rr} 0 & I \\ -I & 0 \end{array} \right) \left( \begin{array}{c} a_1 \\ \vdots \\ a_{2(g-1)} \end{array} \right) $$
Consequently any matrix $\left( \begin{array}{cc} F & E \\ 0 & C \end{array} \right)$ preserving $\langle-,-\rangle$ must satisfy
$$ \left( \begin{array}{cc} F^* & 0 \\ E^* & C^* \end{array} \right)   \left( \begin{array}{rr} 0 & I \\ -I & 0 \end{array} \right) 
			\left( \begin{array}{cc} F & E \\ 0 & C \end{array} \right) =\left( \begin{array}{rr} 0 & I \\ -I & 0 \end{array} \right) $$
This implies $F = (C^*)^{-1}$, and $C^* E = E^* C$.

Theorem \ref{theorem:handlebodytoautFg} implies that $\Lambda$ surjects
onto $\Aut^t(F_g|S)$. Lemma  \ref{lemma:invtsubspace} implies that $pr \circ \rho_i = \eta_i$, and so
Theorem \ref{theorem:GrunLubo} and Remark \ref{remark:stillzardense} finish the proof.
\end{proof}

\section{Generating unipotents in the image}
\label{section:unipotents}
%!TEX root = ./MCGReps.tex
Our goal in this section is to prove that the image of $\rho_{H,p, i}$ contains the finite index subgroups of the unipotent subgroups $\mc N^+(\frak O)$
and $\mc N^-(\frak O)$. We begin in this first subsection
by finding two concrete examples of unipotents in the image of $\rho_{H,p, i}$. 
As before, we fix $H$ and $p$ and drop them from the subscripts of $\rho$ and $\Gamma$.

\subsection{Two unipotent elements}
Let $a, b$ be the simple closed
curves on $\Sigma$ as in Lemmas \ref{lemma:specialtorus} and \ref{lemma:rank2submodule}. The tightly controlled unipotents
are the images of the Dehn twists about $a$ and $b$ which we denote by $T_a$ and $T_b$.

\begin{lemma} \label{lemma:unipotents}
Let $M_i', M_i''$ be the modules and $m_{i,1}', \dots, m_{i, g-1}'$ and $m_{i,1}'', \dots, m_{i, g-1}''$ 
their respective free $A_i$-bases given by Lemma \ref{lemma:niceAibasis}. Then, with respect to the ordered basis 
$m_{i,1}', \dots, m_{i, g-1}', m_{i,1}'', \dots, m_{i, g-1}''$,
$$ \rho_i(T_b) = \left( \begin{array}{cc} I & e_{1,1} \\ 0 & I \end{array} \right) $$
$$  \rho_i(T_a^{-1}) = \left( \begin{array}{cc} I & 0 \\ e_{1,1} & I \end{array} \right) $$
where $e_{1,1} \in \Mat_{g-1}(A_i)$ is the matrix with $1$ in the upper left corner and $0$ in all other
entries.
\end{lemma}

\begin{proof}
We prove it for $T_a^{-1}$; the proof for $T_b$ is similar. Let $\varphi$ be some representative homeomorphism of $T_a$ which
fixes $\ast$. The lift $\tilde \varphi$ is precisely the twist $\varphi$ on each copy $\Sigma_h'$ of $\Sigma$.
I.e., the lift in $\Mod(\tilde \Sigma, \tilde \ast)$ is simply the composition of twists $\prod_{h \in H} T_{a_h}$. Notice that since all the
twists have mutually disjoint support, all the $T_{a_h}$ commute with each other, so the order of composition is irrelevant.
It is a standard fact (see e.g. \cite[Proposition 6.3]{FarbMarg}) that the action on homology $T_{\delta *}$ of a single Dehn twist $T_{\delta}$ on
a homology class $v \in \HH_1(\Sigma, \Q)$ is
$$ T_{\delta *}(v) = v + \langle v, [\delta] \rangle_{\Sp} [\delta]$$
where $[\delta]$ is the homology class of $\delta$.

We now simply compute our action. Notice that $T_{a_h*}([a_{h'}]) = [a_{h'}]$ for all pairs $h \neq h'$ as $a_h, a_{h'}$ are disjoint, and
so we do not ``accumulate'' any extra terms in composing.
I.e. the composition of the $T_{a_h*}$ applied to any  $v \in \HH_1(\tilde \Sigma, \Q)$ is
$$\begin{array}{rcl} \prod_{h \in H} T_{a_h *} (v) &= &v + \sum_{h \in H} \langle v, [a_h] \rangle_{\Sp} [a_h] \\
										    &= &v + \sum_{h \in H} \langle v, h \cdot [a_1] \rangle_{\Sp} h\cdot [a_1] \\
										    &= &v + \langle v, [a_1] \rangle [a_1] \end{array}$$
By the choice of basis, $\langle v, [a_1] \rangle$ is $-1$ when $v = b_1$ and is $0$ when $v$ is any other basis element. Using the inverse Dehn twist then
gives us positive $1$.
\end{proof}

\subsection{The unipotent subgroups $\mc N^+$ and $\mc N^-$}
We again fix some simple component $A = A_i$ of $\Q[H]$ and the corresponding submodule $M = M_i$ of $\hat R$ 
with its associated sesquilinear form. Also let $\rho = \rho_i = \rho_{H, p, i}$.
Let us now consider the following unipotent subgroups of the corresponding group $\mc G = \Aut_A(A^{2g-2}, \langle-,-\rangle)$ 
which we continue to view in terms of the basis from Lemma \ref{lemma:niceAibasis}.
Let us set $B = \Mat_{g-1}(A^{op})$, and for $E = (e_{ij}) \in B$, define $E^*$ to be $(\tau(e_{ji})$. For any 
matrix of the form $\left( \begin{array}{cc} I & E \\ 0 & I \end{array} \right)$ in $\mc G$, a straightforward computation shows that $E^* = E$.
We set
$$\mc N^+ = \left\{ \left( \begin{array}{cc} I & E \\ 0 & I \end{array} \right) \in \Mat_2(B)\; \; | \;\; E^* = E\right\}$$
$$\mc N^- = \left\{ \left( \begin{array}{cc} I & 0 \\ E & I \end{array} \right) \in \Mat_2(B)\; \; | \;\; E^* = E \right\}$$

In this section we show that $\rho(\Gamma)$ contains a lattice in each of these unipotent subgroups. We will produce them by conjugating
the elements $\rho(T_b)$ and $\rho(T_a^{-1})$ by elements in $\rho(\Lambda)$. Recall from Proposition
\ref{prop:handlebodyimage} that $\rho(\Lambda)$ consists of elements of the form 
$\left( \begin{array}{cc} (C^*)^{-1} & E \\ 0 & C\end{array} \right)$ with $E^*C = C^* E$, and note that the inverse of such a matrix is
$\left( \begin{array}{cc} C^* & E^* \\ 0 & C^{-1}\end{array} \right)$. 
One can compute that
$$  \left( \begin{array}{cc} C^* & E^* \\ 0 & C^{-1} \end{array} \right) \left( \begin{array}{cc} I & E' \\ 0 & I \end{array} \right) 
 		\left( \begin{array}{cc} (C^*)^{-1} & E \\ 0 & C\end{array} \right) 
		= \left( \begin{array}{cc} I & C^* E' C \\ 0 & I \end{array} \right)$$
Recall also from Proposition \ref{prop:handlebodyimage} that, if $g \geq 4$, for every $C$ in a finite index subgroup of  $\SL_{m(g-1)}(\mc R)$, the image 
$\rho(\Lambda)$ contains a matrix of the form $\left( \begin{array}{cc} (C^*)^{-1} & E \\ 0 & C \end{array} \right)$ (and for $g=3$, this is true
for every $C$ in some Zariski dense subgroup of $\SL_{m(g-1)}(\mc R)$).
Hence, our interest is in the set of matrices $C^* E C$ where $E = E^*$ and $C \in \SL_{m(g-1)}(\mc R)$.

First, we will understand the action of $\SL_1(B) = \{ C \in B \;\; | \;\; \nrd_{B/L}(C) = 1\}$ on $N = \{ E \in B \;\; | \;\; E^* = E \}$ given by
$C \cdot E = C^* E C$. Recall that $L$ is the center of $A$ and hence $B$, and
that $K$ is the field fixed by $\tau$ and hence by $^*$.
\begin{lemma} \label{lemma:irreducibleSaction}
The above action of $\SL_1(B)$ on $N$ is irreducible over $K$.
\end{lemma}
\begin{proof}
Suppose $\tau$ (hence $^*$) is of the first kind, and choose some isomorphism $\varphi: B \otimes_K \C \to \Mat_n(\C)$ (for appropriate $n$)
so that $^*$ goes over to the involution $\nu$ in (1) or (2) of Lemma \ref{lemma:changeinv}. Applying $\varphi$ to $N \otimes_K \C$
identifies $N \otimes_K \C$ with $N' = \{E \in \Mat_n(\C)\; | \; \nu(E) = E\}$. Let $\varphi'$ be the composition of $B \to B \otimes_K \C$
followed by $\varphi$. Now, $\varphi'(\SL_1(B))$ consists of the $K$-points of a form of $\SL_n$ in $\SL_n(\C)$, and so $\varphi'(\SL_1(B))$
is Zariski dense in $\SL_n(\C)$ by \cite[Theorem A]{BoreSpri}. (See also \cite{Grot} or \cite{Rose}.)
Thus, if $\SL_n(\C)$ acts irreducibly on $N'$, so does
$\varphi'(\SL_1(B))$, and so $\SL_1(B)$ acts irreducibly on $N$. We will show that for any nonzero $E \in N'$, the $\SL_n(\C)$-orbit
spans $N'$. Since scaling $E$ does not affect the span, it is equivalent to prove this for the $\GL_n(\C)$-orbit. 

In case (1), we have $\nu(C) = C^t$, and $N'$ is the set of symmetric matrices and the action is $C \cdot E = C^t E C$.
This action is well-known to be irreducible.
In case (2), we have $\nu(C) = J C^t J^{-1}$. Notice that $J^{-1} = J^t = -J$. The set $N'$ consists of those matrices $E$ such that
$J E^t J^{-1} = E$ which is equivalent to $J^{-1} E$ being skew-symmetric. Then, we have
$$ J^{-1} (C \cdot E) = J^{-1} J C^t J^{-1} E C  = C^t (J^{-1}E) C$$
Consequently, this case reduces to the fact that the action $E \to C^t E C$ on the set of skew-symmetric matrices $E$ is
irreducible. The irreducibility of this action is also well-known.

Suppose now that $\tau$ is of second kind. We use a proof very similar to the previous cases. 
By Lemma \ref{lemma:changeinv}, there is an isomorphism
$\varphi: B \otimes_K \C \to \Mat_n(\C) \times \Mat_n(\C)$ (for appropriate $n$), so that $^*$ goes over to
the involution $\nu$ defined by $\nu(C, D) = (D^t, C^t)$. Thus, the homomorphism $\varphi$ identifies $N \otimes_K \C$ with
$N' = \{ (E, E^t) \in \Mat_n(\C) \times \Mat_n(\C) \}$.
Let $\varphi'$ be the composition of $B \to B \otimes_K \C$ followed by $\varphi$. By \cite[Theorem A]{BoreSpri},
the image $\varphi'(\SL_1(B))$ is Zariski dense in $\SL_n(\C) \times \SL_n(\C)$. Just as before, $\SL_1(B)$
acts irreducibly on $N$ if $\SL_n(\C) \times \SL_n(\C)$ acts irreducibly on $N'$. In this case, the action on $N'$ is
$$ (C, D) \cdot (E, E^t) = (C^t, D^t) (E, E^t) (D, C) = (C^t E D, D^t E^t C).$$
It is clearly irreducible.

\end{proof}

We can now show that our image contains a finite index subgroup of $\mc N^+(\frak O) = \mc G(\frak O) \cap \mc N^+$.

\begin{lemma} \label{lemma:getU+}
The intersection $\rho(\Lambda) \cap \mc N^+(\frak O)$ is of
finite index in $\mc N^+(\frak O)$.
\end{lemma}
\begin{proof}
Let $\mc F = \rho(\Lambda) \cap \mc G(\frak O)$.
We show that the subgroup of the (abelian) group $\mc N^+(\frak O)$ generated by conjugates of 
$\left( \begin{array}{cc} I & e_{1,1} \\ 0 & I \end{array} \right)$ by $\mc F$ is of finite index in $\mc N^+(\frak O)$.
(Note that by construction, this unipotent matrix lies in $\rho(\Lambda)$.)
Using Proposition \ref{prop:handlebodyimage}, this is equivalent to the subgroup of $N(\frak O)$ generated by
the orbit of $e_{1,1}$ under the action of the finite index subgroup $G = pr(\rho(\Lambda))$ of $\SL_{m(g-1)}(\mc R)$.
By Proposition \ref{prop:handlebodyimage},
the group $G$ is Zariski dense in $\SL_1(B) \cong \SL_{g-1}(A^{op})$. Consequently, since $\SL_1(B)$
acts irreducibly by Lemma \ref{lemma:irreducibleSaction}, so does $G$, and thus the group generated by the $G$-orbit of $e_{1,1}$ 
is of finite index in $\mc N^+(\frak O)$.
\end{proof}

Using this lemma, we can now see that $\rho(\Lambda)$ contains block diagonal matrices.
In fact, we can show the following.
\begin{lemma} \label{lemma:blockdiagonal} Assume $g \geq 4$.
The image $\rho(\Lambda)$ contains a finite index subgroup of the following group:
$$\left\{ \left(\begin{array}{cc} (C^*)^{-1} & 0 \\ 0 & C \end{array} \right) \in \Mat_{2}(B)\; | \; C \in \SL_{m(g-1)}(\mc R) \right\} $$
where $\mc R$ is the order from Theorem \ref{theorem:GrunLubo}.
\end{lemma}
\begin{proof}
Let $\mc F = \rho(\Lambda) \cap \mc G(\frak O)$ which necessarily lies in $\mc P(\frak O) = \mc P \cap \mc G(\frak O)$ by
Proposition \ref{prop:handlebodyimage}. Recall that unipotent subgroups have the congruence subgroup property, so, by Lemma \ref{lemma:getU+},
there is some ideal $\frak a$ of $\frak O$ such that $\rho(\Lambda) \cap \mc N^+(\frak O)$ contains $\mc N^+(\frak a)$ (see e.g. page 303 of \cite{Ragh2}).
After passing to a finite index subgroup $\mc F'$ of $\mc F$, we can ensure that
if $\left( \begin{array}{cc} (C^*)^{-1} & E \\ 0 & C \end{array} \right) \in \mc F'$, then $\JMat{ I & E \\ 0 & I } \in \mc N^+(\frak a)$ 
and therefore $\left(\begin{array}{cc} I & - C^* E  \\ 0 & I \end{array} \right)$ also lies in 
$\mc N^+(\frak a)$. The result then follows
from Proposition \ref{prop:handlebodyimage} and the computation:
$$  \left( \begin{array}{cc} (C^*)^{-1} & E \\ 0 & C \end{array} \right)  \left(\begin{array}{cc} I & -C^* E \\ 0 & I \end{array} \right)  = 
	\left(\begin{array}{cc} (C^*)^{-1} & 0 \\ 0 &  C \end{array} \right)$$
\end{proof}

\begin{remarknum} \label{remark:blockdiagonal} For $g=3$, a similar argument using Remark \ref{remark:stillzardense} instead of Proposition \ref{prop:handlebodyimage}
will show that $\rho(\Lambda)$ contains a Zariski dense subgroup of 
$$\left\{ \left(\begin{array}{cc} (C^*)^{-1} & 0 \\ 0 & C \end{array} \right) \in \Mat_{2}(B)\; | \; C \in \SL_{m(g-1)}(\mc R) \right\}. $$
\end{remarknum}

The block diagonal matrices act in similar ways on $\mc N^+(\frak O)$ and $\mc N^-(\frak O)$ by conjugation. Applying the argument of 
Lemma \ref{lemma:getU+} but using the block diagonal matrices from Lemma \ref{lemma:blockdiagonal} and Remark \ref{remark:blockdiagonal}
in place of $\rho(\Lambda) \cap \mc G(\frak O)$, we obtain the following lemma.
\begin{lemma} \label{lemma:getU-}
The intersection $\rho(\Gamma) \cap \mc N^-(\frak O)$ is of finite index in $\mc N^-(\frak O)$.
\end{lemma}

\section{The image of the mapping class group}
\label{section:image}
%!TEX root = ./MCGReps.tex
We are now ready to finish the proof of our main technical result Theorem \ref{theorem:mainresult}, 
and then to deduce Theorem \ref{theorem:maintheoremprocedural} which outlines the general procedure of obtaining arithmetic quotients of the mapping class group.
In the next section, we will apply the procedure to specific examples of $H$ and $r$ to deduce Theorem \ref{theorem:examplesofquotients},
its corollaries, and Theorem \ref{theorem:examplesofquotients2}.

\subsection{Finishing the proof of Theorem \ref{theorem:mainresult}}
First, let $\Gamma = \Gamma_{H,p}, \; \rho = \rho_{H,p}, \; \mc G = \mc G_{H,i}$, and $\mc G^1 = \mc G_{H,i}^1$.
We can now use the results of Lemmas \ref{lemma:getU+}, \ref{lemma:blockdiagonal}, \ref{lemma:getU-}
to get the following.
Let $\mc G$ be the algebraic group $\Aut_A(A^{2g-2}, \langle-,-\rangle)$. It is a reductive group but
not necessarily semisimple. Let $\mc G^1$ be the elements of reduced norm $1$ over $L$.
This subgroup $\mc G^1$ is semisimple.

First, let us see why a finite index subgroup of $\Gamma$ must map into 
$\mc G(\frak O) = \Aut_{\frak O}(\frak O^{2g-2}, \langle-,-\rangle)$.
Recall that $\rho$ is originally defined by the action of $\Gamma < \Aut(T_g)^+$ on $\overline{R} = R/[R,R]$
which is induced by the action of $\Gamma$ on $R < T_g$. After tensoring $\overline{R}$ by $\Q$
we obtain $\hat{R} \cong \Q[H]^{2g-2} \oplus \Q^2$ and the action of $\Gamma$ projects
to an action of $A^{2g-2}$ preserving $\langle-,-\rangle$. Since $\Gamma$ preserves
$\overline{R} \subset \hat{R}$, it also preserves $\overline{R} \cap A^{2g-2}$ which is
a lattice in the $K$-vector space $A^{2g-2}$. Thus, some finite index subgroup of $\Gamma$ maps into $\mc G(\frak O)$.

Recall that $\mc P$ is the parabolic subgroup of $\mc G$ which consists of those elements preserving
the submodule $M'$ in the notation
of Section \ref{section:isotropic}. Let $\mc P^1 = \mc P \cap \mc G^1$. Then $\mc P^1$ is the normalizer of 
the unipotent subgroup $\mc N^+$. The group $\mc P^1$ is a semi-direct product of
its Levi factor $\mc L$ and its unipotent radical, which is equal to $\mc N^+$.
The Levi factor $\mc L$ is isomorphic to $\GL_{g-1}(A^{op}) \cong \GL_{(g-1)m}(D^{op})$. Let
$T$ be the maximal $K$-split torus of $\mc L$ given by diagonal matrices with entries in $K$; in particular,
$T$ is isomorphic to $(K^*)^{m(g-1)}$. Viewing $\mc L$ as block diagonal matrices, $T$ is also a maximal $K$-split
torus for $\mc G^1$. I.e. the following is a maximal $K$-split torus for $\mc G^1$:
$$ \{\JMat{ (C^*)^{-1} & 0 \\ 0 & C } \;\; | \;\; C \in T \}.$$
(See classification in \cite[Table II]{Tits}.) Let $\Phi$ be a $K$-root system with respect to this
torus. Then for every $\alpha \in \Phi$, the corresponding root subgroup lies either
in $\mc N^+, \mc N^-$, or $\mc L$.

Now, when $g \geq 4$, Lemmas \ref{lemma:getU+}, \ref{lemma:blockdiagonal}, \ref{lemma:getU-} imply, therefore,
that for every $\alpha \in \Phi$, the intersection of $\rho(\Gamma)$ with the root subgroup is commensurable
to the integral points of the $K$-root subgroup. Recall also that unipotent subgroups have an affirmative
answer to the congruence subgroup problem \cite{Ragh2}, and hence a finite index subgroup of the root subgroup
contains a congruence subgroup of it. For $g = 3$, we can still get the same conclusion even though
Lemma \ref{lemma:blockdiagonal} is not valid, as explained in Remark \ref{remark:stillzardense}.

We are now in a position to use Theorem 1.2 of \cite{Ragh} which asserts that exactly in such a situation
the group generated by congruence subgroups of $U_\alpha$ as $\alpha$ ranges over $\Phi$ is an arithmetic group, i.e.
of finite index in $\mc G^1(\frak O)$. (The theorem requires that the $\Q$-rank of $\mc G^1$ is at least
$2$, and this is always true when $g \geq 3$.)
We thus deduce that $\rho(\Gamma)$ contains a finite index subgroup of
$\mc G^1(\frak O)$. By Proposition \ref{prop:GOfiniteindex}, $\mc G^1(\frak O)$ is of finite index
in $\mc G(\frak O)$, and so $\rho(\Gamma)$ and $\mc G^1(\frak O)$ are commensurable.

\subsection{Passing to mapping class groups without fixed point} \label{section:forgetmarkedpt}
We have so far only proven results for $\Mod(\Sigma, \ast)$. We now show how the homomorphism $\rho$
induces a map on $\Mod(\Sigma)$. Let $\epsilon: \Mod(\Sigma, \ast) \to \Mod(\Sigma)$ be the natural
map that forgets the fixed point. The following proposition says that after passing to a finite index
subgroup, $\rho$ factors through $\epsilon$.

\begin{prop} \label{proposition:forgetmarkedpoint}
Let $p: T_g \to H$ be a $\phi$-redundant homomorphism and 
$\rho_{H,p}: \Gamma_{H,p} \to \Aut_{\Q[H]}(\hat R, \langle -, - \rangle)$ the corresponding
homomorphism. Then, there is a finite index subgroup $\Gamma'_{H,p} < \Gamma_{H,p}$
so that $\rho_{H,p}|_{\Gamma'_{H,p}}$ factors through $\epsilon$.
\end{prop}

\begin{proof}
We recall the Birman Exact Sequence for $\Mod(\Sigma, \ast)$: (\cite{FarbMarg, Birm})
$$ 1 \to T_g \overset{c}{\to} \Mod(\Sigma, \ast) \overset{\epsilon}{\to} \Mod(\Sigma) \to 1. $$
Here $T_g$ maps to the ``point-pushing'' mapping classes and each $c(\alpha)$ for $\alpha \in T_g$
acts as conjugation by $\alpha$ on the fundamental group. For $\rho$ to factor through
$\epsilon$ on some subgroup $\Gamma' < \Mod(\Sigma, \ast)$, we need 
$\epsilon(\Gamma' \cap c(T_g))$ to act trivially on $R/[R,R]$. Hence,
it suffices that $\Gamma' \cap c(T_g) \subseteq c(R)$.

Let $\alpha \in T_g$ and suppose $c(\alpha) \in \Gamma_{H,p}$.
Since $R$ is of finite index, there is some $n$ such that $\alpha^n \in R$,
and so $c(\alpha)^n$ acts trivially on $R/[R,R]$ by conjugation. Thus,
$\rho(\Gamma_{H,p} \cap c(T_g))$ consists entirely of torsion elements.
The image $\rho(\Gamma_{H,p})$ lies in $\Aut(\overline{R}, \langle -,-\rangle_{\Sp})$
which is isomorphic to $\Sp(2 + (2g-2) |H|, \Z)$ and contains some finite
index torsion-free subgroup $G$. Thus, $\Gamma'_{H,p} = \rho^{-1}(G) \cap \Gamma_{H,p}$
is the required subgroup.
\end{proof}

\subsection{Producing virtual quotients from irreducible representations of finite groups}
Theorem \ref{theorem:mainresult} and Section \ref{section:forgetmarkedpt} establish Theorem
\ref{theorem:maintheoremprocedural} which presents a procedure
to obtain arithmetic quotients and which we summarize now. This theorem will be convenient
for use in Section \ref{section:finitegroups} and possibly the future.

Let $\Sigma_g$ be a closed surface, 
let $H$ be some finite group generated by $d(H) < g$ generators, and let 
$A$ be a simple component of $\Q[H]$. This is the required input for the theorem. By Lemma \ref{lemma:algdecomp},
the standard involution $\tau$ on $\Q[H]$ defined by $h \mapsto h^{-1}$ restricts to an involution
on $A$ which, by abuse of notation, we will also call $\tau$. 

Set $M = A^{2g-2}$ with 
free basis $x_1, \dots, x_{g-1}, y_1, \dots y_{g-1}$ and let $\langle-,-\rangle$
be the skew-Hermitian form sesquilinear relative to $\tau$ such that
$\langle x_i, y_j \rangle = \delta_{ij}, 
\langle y_i, x_j \rangle = -\delta_{ij}, \langle x_i, x_j \rangle = 0,$ and $\langle y_i, y_j \rangle = 0$.
Let $\mc G = \Aut_A(A^{2g-2}, \langle-,-\rangle)$, and set 
$\Omega_{H, A} = \mc G^1(\frak O) = \mc G^1 \cap \Aut_{\frak O}(\frak O^{2g-2}, \langle-,-\rangle)$ where
$\frak O$ is the order in $A$ which is the image of $\Z[H]$.

To obtain our representation, we need some $\phi$-redundant homomorphism $p: T_g \to H$.
Let $\phi: T_g \to F_g$ be the map on fundamental groups induced by the inclusion
$\Sigma_g \hookrightarrow \mc H_g$. Since $d(H) < g$, there is some epimorphism $p': F_g \to H$
which maps at least one free generator to the trivial element. The composition $p = p' \circ \phi$
is $\phi$-redundant. Recall that $R$ is the kernel of $p$ and $\overline{R} = R/[R,R]$. From
the discussion in Section \ref{section:liftrep}, the following subgroup of $\Mod(\Sigma_g, \ast) \cong \Aut(T_g)^+$
$$\Gamma = \{f \in \Aut(T_g)^+ \; | \;  p \circ f = p \}$$
has a well-defined action on $\overline{R}$, hence on $\hat{R} \cong \Q[H]^{2g-2} \oplus \Q^2$,
and hence on $A^{2g-2}$. This induces a homomorphism $\Gamma \to \mc G$ by Lemmas
\ref{lemma:rhotomoduleaut} and \ref{lemma:liftpreserveform}.
By Theorem \ref{theorem:mainresult}, there is a finite index subgroup $\Gamma' < \Gamma$
and a homomorphism $\rho: \Gamma' \to \Omega_{H,A}$ with finite index image. By 
Lemma \ref{proposition:forgetmarkedpoint}, there is a finite index subgroup $\Gamma_{H,A} < \Mod(\Sigma_g)$
and a representation $\rho_{H,A}: \Gamma_{H,A} \to \Omega_{H,A}$ with finite index image.
This establishes the theorem.

\subsection{Some remarks on the uniqueness of $\rho_{H,r}$} \label{subsection:rhouniqueness}
Notice that in the above $\Omega_{H,r} = \Omega_{H, A}$ is uniquely determined by $H$ and $r$, but
the representation $\rho_{H,r} = \rho_{H,A}$ depends on the choice of the surjective representation $p: T_g \to H$. One way
to get other epimorphisms $p$ is via the action of the group $\Aut(T_g) \times \Aut(H)$
by $(\psi, \varphi) \cdot p = \varphi \circ p \circ \psi^{-1}$. The corresponding induced
representations $\rho$ are ``equivalent'' in the following sense. Let $\inn(\psi): \Aut(T_g) \to \Aut(T_g)$
denote the inner automorphism $\inn(\psi)(f) = \psi f \psi^{-1}$.
The proof of the following lemma is elementary, so we omit it.

\begin{lemma} \label{lemma:equivalentrhos}
Suppose $p_1: T_g \to H$ is some epimorphism and $p_2 = \varphi p_1 \psi^{-1}$ for some
$(\psi, \varphi) \in \Aut(T_g) \times \Aut(H)$. Let $R_i = \ker(p_i)$,let  $\hat{R}_i = (R_i/[R_i,R_i]) \otimes_\Z \Q$, and
let $\rho_{H, p_i}: \Gamma_{H, p_i} \to \Aut_{\Q[H]}(\hat{R}_i, \langle-,-\rangle)$ be the
corresponding representation as presented in Section \ref{section:liftrep}.
Then, there is an isomorphism 
$\epsilon: \Aut_{\Q[H]}(\hat{R}_1) \to \Aut_{\Q[H]}(\hat{R}_2)$ such that $\rho_{H, p_2} = \epsilon \circ \rho_{H, p_1} \circ \inn(\psi)$.
\end{lemma}

Recall that our maps $T_g \to H$ are $\phi$-redundant, so a natural question is: how
transitively does $\Aut(T_g) \times \Aut(H)$ act on them?
Suppose $p_1, p_2: T_g \to H$ are respectively $\phi_1$-redundant and $\phi_2$-redundant, and
$p_1', p_2': F_g \to H$ are the respective induced maps.
Recall from the introduction and Section \ref{section:isotropic} that all maps $\phi: T_g \to F_g$ are induced by some identification $\Sigma \cong \partial \mc H$.
This fact combined with Theorem \ref{theorem:handlebodytoautFg} implies that $p_1$ and $p_2$ are
are in the same $\Aut(T_g) \times \Aut(H)$-orbit if $p_1'$ and $p_2'$ are in the same
$\Aut(F_g) \times \Aut(H)$-orbit.

The question then reduces to the transitivity of the $\Aut(F_g) \times \Aut(H)$ action
on redundant homomorphisms $F_g \to H$. In \cite[Conjecture 6.3]{Lubo}, it was conjectured
that this action is transitive. This is known to be true when $H$ is solvable \cite{Dunw} and when $H$ is simple
\cite{Gilm, Evan} (see also \cite[Theorem 6.6]{Lubo}). If the conjecture is true, then our homomorphisms
$\rho_{H,r}$ depend only on $(H,r)$ up to the above equivalence.

A more general problem is to understand the $\Aut(T_g) \times \Aut(H)$-orbits 
of all epimorphisms, $\Epi(T_g, H)$, not just those which are $\phi$-redundant. One obstruction to transitivity is
$\HH_2(H, \Z)$. Given an epimorphism $T_g \to H$, there is an induced
map $\HH_2(T_g, \Z) \to \HH_2(H, \Z)$ which specifies a well-defined element
of $\HH_2(H, \Z)/\Out(H)$ up to the $\Aut(T_g) \times \Aut(H)$-action.
Those epimorphisms $T_g \to H$ factoring through $F_g$ and, in particular,
$\phi$-redundant homomorphisms induce the ($\Out(H)$-orbit of the) trivial element in $\HH_1(H, \Z)/\Out(H)$.
Theorem 6.20 of \cite{DunfThur} shows that for any fixed finite group $H$
and sufficiently large $g$, the $\Aut(T_g) \times \Aut(H)$-orbits of $\Epi(T_g, H)$
are in bijective correspondence with $\HH_2(H, \Z)/\Out(H)$. 
We make the following conjecture.

\begin{conjecture} If $H$ is finite and $g > d(H)$, then the $\Aut(T_g) \times \Aut(H)$-orbits
of $\Epi(T_g, H)$ are in one-to-one correspondence with $\HH_2(H, \Z)/\Out(H)$.
\end{conjecture}

Note that for any surjective homomorphism $p:T_g \to H$ and irreducible
$\Q$-representation $r$ of $H$, there is an associated representation $\rho_{H,r}$
even if $p$ is not $\phi$-redundant. The assumption that $p$ is $\phi$-redundant is
only necessary for us to establish that the image is arithmetic.
If the above conjecture is true, then, up to the natural equivalence, $\rho_{H,r}$ is
uniquely determined by the pair $(H, r)$ and an element of $\HH_2(H, \Z)/\Out(H)$.

\subsection{Remarks on property $(T)$ and Theorem \ref{theorem:mainresult} for the case $g = 2$} 
\label{subsection:propT}
Another remark which seems to be worth mentioning is that, except for the case $g=2$, all
the virtual arithmetic quotients we obtain for $\Mod(\Sigma_g)$ are lattices in higher-rank
semisimple Lie groups (whose simple factors are of $\R$-rank $\geq 2$) and hence have Kazhdan's property $(T)$.
We do not know if this is the case for the analogous virtual quotient of $\Aut(F_n)$ constructed
in \cite{GrunLubo}. There it is possible that the image has an infinite abelian quotient. See
\cite{GrunLubo} for a detailed discussion.

\subsection{Image of subgroups in the Johnson filtration} \label{subsection:johnson}
As promised in the introduction, we show that the groups in the Johnson filtration
are ``very different'' from the new Torelli groups we define here. Precisely, we prove the following.
\begin{prop} \label{prop:torelliimage}
Assume $g \geq 3$.
Let $H$ be a nontrivial finite group, $r$ a nontrivial irreducible $\Q$-representation of $H$,
and let $\Gamma=\Gamma_{H, r}$ and $\rho = \rho_{H,r}$ be as in 
Theorem \ref{theorem:maintheoremprocedural}. Let $\frak T = \frak T_{H, r} = \ker(\rho_{H,r})$
and $\mc I$ any element in the Johnson filtration of $\Mod(\Sigma)$. Then $\mc I \cdot \frak T$
is of finite index in $\Mod(\Sigma)$.
\end{prop}
Before proving the proposition, note that this implies $\rho_{H,r}(\mc I)$ has finite index image
in $\Omega_{H,r}$ when $H$ is nontrivial. Thus, our main results imply also that the classical
Torelli group and the subgroups in the Johnson filtration have a rich collection of arithmetic quotients.

\begin{proof}
First, suppose that $\mc I$ is the classical Torelli group, i.e. the kernel of the standard homomorphism
$\Mod(\Sigma) \to \Sp(2g, \Z)$. Let $\mc I' = \mc I \cap \Gamma$. Then $\mc I'$ and $\frak T$ are both
normal in $\Gamma$. Consider $\mc I \cdot \frak T/\mc I$ which is a normal subgroup of $\Gamma/\mc I$,
a finite index subgroup of $\Sp(2g, \Z)$. By Margulis' normal subgroup theorem \cite[Theorem (4')]{Marg},
it is either finite or of finite index. In the latter case, it follows that $\mc I \cdot \frak T$ is of finite index in
$\Mod(\Sigma)$.

Suppose then that $\mc I \cdot \frak T/\mc I$ is finite. Then $\hat{\frak T} = \mc I \cap \frak T$ is of finite
index in $\frak T$. Consider now $\hat{\frak T} \mc I / \hat{\frak T}$ in $\Gamma/\hat{\frak T}$. The latter group
is commensurable to a higher rank lattice, so as above $\hat{\frak T} \mc I / \hat{\frak T}$ is finite or of finite index. If the latter holds,
we are done. If the former holds, then $\mc I \cap \hat{\frak T}$ is of finite index in $\mc I$, and so 
$\frak T$ and $\mc I$ are commensurable, but that is impossible.

Now, let $\mc I$ be any group in the Johnson filtration and let $\mc I_0$ be the classical Torelli group.
Let $\mc I_0' = \Gamma \cap \mc I_0$ and $\mc I' = \Gamma \cap \mc I$.
The quotient $\mc I_0/\mc I$ is nilpotent, and thus $\mc I_0'/\mc I'$ is nilpotent. Let $\Pi: \Gamma \to \Gamma/\frak T$
be the quotient map.
The above two paragraphs show that $\Pi(\mc I_0')$ is of finite index in $\Gamma/\frak T$. Thus,
$\Pi(\mc I_0')/\Pi(\mc I')$ is a nilpotent quotient of a higher rank lattice and hence finite. 
This implies that $\mc I \cdot \frak T$ is of finite index in $\Mod(\Sigma)$.
\end{proof}

\section{Finite groups and arithmetic quotients}
\label{section:finitegroups}
%!TEX root = ./MCGReps.tex
As explained in the introduction, our main result gives us a way to associate a virtual epimorphism $\rho$ from $\Mod(\Sigma)$
to an arithmetic group $\Omega$ from the choice of a finite group $H$ and irreducible $\Q$-representation $r$ or, equivalently,
a simple component $A$ of $\Q[H]$. In this section,
we present various kinds of arithmetic groups which arise as we vary $H$ and $r$. 

In the first section, we prove Corollary \ref{corollary:mods2large}. While this is formally a corollary of Theorem \ref{theorem:examplesofquotients}
as stated, our actual proof must proceed by deducing the corollary directly. 
Then, for genus $2$, the various examples in
Theorem \ref{theorem:examplesofquotients} actually follow from the ``corollary''. Note, then, that for genus $2$, we do
{\em not} obtain these virtual arithmetic quotients of $\Mod(\Sigma_2)$ by various representations $\rho_{H, p}$.
In the next few sections, we prove Theorem \ref{theorem:examplesofquotients} by listing the finite groups $H$ and irreducible
representations $r$,
establishing the requisite properties and determining the arithmetic group $\Omega_{H,A}$ with the use of results in Section
\ref{section:tautypekind} when necessary. We then prove Corollary \ref{corollary:surjectfinitegroups} using example (a) of Theorem \ref{theorem:examplesofquotients}.
In the final section, we discuss more generally the question of what simple $\Q$-algebras with involution can arise as factors of 
$\Q[H]$ and deduce Theorem \ref{theorem:examplesofquotients2}. 

\subsection{Proof of Corollary \ref{corollary:mods2large} and Theorem \ref{theorem:examplesofquotients} for $g = 2$}
We set $H = \Z/2\Z$ and we let $p: T_2 \to \Z/2\Z$ be some surjective $\phi$-redundant map. Here, $\Q[H] \cong \Q \oplus A$
where $A \cong \Q$ as a $\Q$-algebra and where as a $\Q[H]$ module, the nontrivial element of $H$ acts on $A$ by
multiplication by $-1$. Note that the standard involution $\tau$ on $\Q[H]$ is trivial, and hence its restriction to $A$ is trivial. 
Thus, $\Aut(A^{2g-2}, \langle-,-\rangle) = \Sp(2, \Q) = \SL(2, \Q)$. 
From Lemma \ref{lemma:unipotents}, it follows directly that $\rho(\Gamma)$ contains the generating elements 
$$\left(\begin{array}{rr} 1 & 1 \\ 0 & 1 \end{array} \right) \;\;\;\;\;\;\;\left(\begin{array}{rr} 1 & 0 \\ 1 & 1 \end{array} \right). $$
Moreover, as $\rho(\Gamma_{H,p})$ preserves $\overline{R} \subseteq \hat{R}$, there is some finite index subgroup of $\Gamma_{H,p}$
whose image is $\Sp(2, \Z)$.

Because of the corollary, $\Mod(\Sigma_2)$ virtually surjects onto a finitely generated free group, and thus onto all
free groups of finite rank. Since all the arithmetic groups in Theorem \ref{theorem:examplesofquotients} are finitely generated,
$\Mod(\Sigma_2)$ virtually surjects onto them.

\subsection{The symmetric and alternating groups} \label{section:symgroup}
In this section, we describe an irreducible representation for the symmetric group and alternating group.
Let $\Sym(m+1)$ denote the symmetric group on $m+1$ letters and $\Alt(m+1)$ the alternating group on $m+1$
letters. The symmetric group has an irreducible representation
(sometimes called the standard representation) defined as follows. There is an obvious action of $\Sym(m+1)$ on $\Q^{m+1}$ acting
by the permutation matrices. The vector $w = (1, 1, \dots, 1)$ is a fixed vector, but $\Sym(m+1)$ acts irreducibly on the orthogonal complement
$V = w^\perp$. The representation $r: \Sym(m+1) \to \GL(V)$ is irreducible and moreover, the induced homomorphism
$\tilde r: \Q[\Sym(m+1)] \to \End_\Q(V)$ is surjective, and so the corresponding simple factor in $\Q[\Sym(m+1)]$ is
$A \cong \Mat_m(\Q)$. Since the center is $L = \Q$, the fixed field $K$ under $\tau$ is $K = L = \Q$. Since 
$A \otimes_K \R = A \otimes_\Q \R \cong \Mat_m(\R)$, by Proposition \ref{prop:typefromrealrep}, $\tau$ is of first kind and orthogonal type.

For $m \geq 3$, the restriction $r : \Alt(m+1) \to \GL(V)$ is irreducible, so $A = \Mat_m(\Q)$ is a simple component
of $\Q[\Alt(m+1)]$ also. For the same reasons as above, the induced involution $\tau$ on $A$ is of first kind and orthogonal type.
For all $m \geq 2$, a $3$-cycle $\gamma$ and $m$-cycle $\delta$ for $m$ odd (resp. $m+1$-cycle $\delta$ for $m$ even)  
generate $\Alt(m+1)$; moreover, we can choose $\delta, \gamma$ such that $\Alt(m+1)$ is generated by all elements of the form
$\delta^k \gamma \delta^{-k}$.

\subsection{Example (a) of Theorem \ref{theorem:examplesofquotients}}
In this section, we establish example (a) of Theorem \ref{theorem:examplesofquotients} for genus $g \geq 3$. 
Here and in the following three sections, we use our main result, Theorem \ref{theorem:maintheoremprocedural}, 
and so here, and in the following sections, we generally present the following:
\begin{itemize}
\item the finite group $H$ and a demonstration that $d(H) \leq 2 < g$
\item an irreducible representation $r$ of $H$ and/or the corresponding simple factor $A$ of $\Q[H]$
\item an identification of a finite index subgroup of the arithmetic group $\Omega_{H, A}$.
\end{itemize}

For example (a), let $H = \Sym(m+1)$. It is well-known that $\Sym(m+1)$ is generated by two elements.
We take $r$ to be the standard representation, so the corresponding simple factor $A$ is
$\Mat_m(\Q)$ as explained in Section \ref{section:symgroup}.

We now describe $\Omega_{H, A}$.
Let $\mc G = \Aut_{A}(A^{2g-2}, \langle -, -\rangle)$, and let $\sigma$ be the involution of $\End_A(A^{2g-2})$ associated to $\langle -,-\rangle$.
Since $\tau$ is of orthogonal type, the involution $\sigma$ associated to $\langle-,-\rangle$ is of symplectic type by Lemma \ref{lemma:oppositetype}.
The endomorphism ring $\End_A(A^{2g-2})$ is isomorphic to $\Mat_{2(g-1)}(A^{op}) \cong \Mat_{2m(g-1)}(\Q)$.
Over fields $K$ of characteristic $0$, all symplectic involutions of $\Mat_{2(g-1)}(K)$ are equivalent, and so there is an isomorphism
$\varphi: \End_A(A^{2g-2}) \to \Mat_{2m(g-1)}(\Q)$ sending $\sigma$ to the standard symplectic involution. The intersection of 
$\varphi(\mc G^1(\frak O))$ and $\Sp(2m(g-1), \Z)$ is of finite index in the latter. 
By Theorem \ref{theorem:maintheoremprocedural}, there is a virtual epimorphism of $\Mod(\Sigma_g)$ onto $\Omega_{H,A} = \mc G^1(\frak O)$
which, up to finite index, is isomorphic to $\Sp(2m(g-1), \Z)$.
This establishes example (a) in Theorem \ref{theorem:examplesofquotients}.

\subsection{Example (b) of Theorem \ref{theorem:examplesofquotients}} \label{section:exampleb}
The required finite group will be $H = \Dih(2n)$ for $m = 1$ and $H = \Alt(m+1) \times \Dih(2n)$ for $m \geq 3$
where $\Dih(2n)$ is the dihedral group of order $2n$ which has presentation
$$\Dih(2n) = \langle x, y \;\; | \;\; x^n = 1,\;  y^2 = 1,\; y x y^{-1} = x^{-1} \rangle$$
For $m = 2$, we require
a group containing $\Dih(2n)$ which we will describe below. (We remark that for
$g \geq 4$, one can use $H = \Sym(m+1) \times \Dih(2n)$ for all $m$; however this choice is not suitable
for $g=3$ as the group is not generated by two elements for $n$ even.)

We first establish (b) in the case of $m \geq 3$. The group $H$ is generated by
$(\gamma, y)$ and $(\delta, x)$ for the following reasons. Both $(\gamma, 1)$
and $(1, y)$ are powers of $(\gamma, y)$. By conjugating $(\gamma, 1)$ by powers of 
$(\delta, x)$, we obtain all elements of the form $(\delta^k \gamma \delta^{-k}, 1)$. From
this, we generate $\Alt(m+1) \times 1$, and then clearly $(\gamma, y)$ and $(\delta, x)$
generate $H$.

We next describe an irreducible representation of $\Dih(2n)$.
Let $\zeta$ be a primitive $n$th root of unity, and let $\alpha \to \overline{\alpha}$ denote the order $2$ automorphism
of $\Q(\zeta)$ defined by $\zeta \mapsto \zeta^{-1}$. Set
$$B = \{ \left(\begin{array}{rr} \alpha & \beta \\ \overline{\beta} & \overline{\alpha} \end{array} \right) 
		\; \; | \;\; \alpha, \beta \in \Q(\zeta) \}. $$
The group ring $\Q[H]$ surjects onto the $\Q$-algebra $B$ via the homomorphism $\tilde s$ defined by 
$$\tilde s(x) = \left( \begin{array}{rr} \zeta & 0 \\ 0 & \overline{\zeta} \end{array} \right) \;\;\;\;
	\tilde s(y) = \left( \begin{array}{rr} 0 & 1 \\ 1 & 0 \end{array} \right).$$
The center $L$ of $B$ is equal to the subfield of $\Q(\zeta)$ invariant under $\overline{\phantom{a}}$,
i.e. $L = \Q(\zeta)^+$. Recall that the induced involution, call it $\nu$, on $B$ is the unique involution satisfying
$\nu(\tilde s(z)) = \tilde s(z)^{-1}$ for all $z \in \Dih(2n)$. It follows that the involution is the one given by
conjugate transpose of the matrix, and so $\nu$ is of first kind. Moreover, $\dim_L(B) = 4$ and 
the $\nu$-fixed vector subspace is $3$-dimensional, so $\nu$ is of orthogonal type.

We can describe $B$ further. Recall that $B$ is a matrix ring over a division algebra with center
$L$. Notice, however, that $B$ is not a division algebra, and since it has dimension $4$ over $L$,
it follows that $B \cong \Mat_2(L)$.

We now describe the simple factor $A$ of $\Q[H]$ and the involution $\tau$.
Let $\tilde s': \Q[\Alt(m+1)] \to \Mat_m(\Q)$ be the standard representation
from Section \ref{section:symgroup}, and let $\nu'$ be the involution of $\Mat_m(\Q)$ induced by the standard involution
of $\Q[\Alt(m+1)]$. Set $A = \Mat_m(\Q) \otimes_\Q B$.
There is a surjective homomorphism $\tilde r: \Q[H] \to A$ defined by 
$\tilde r(z, w) = \tilde s'(z) \otimes \tilde s(w)$ on $(z,w) \in H =\Alt(m+1) \times \Dih(2n)$.
The involution $\tau$ of $A$ induced by the standard involution of $\Q[H]$ satisfies
$$\tau( x \otimes b) = \nu'(x) \otimes \nu(b).$$
It is straightforward to check that $\tau$ is of first kind.
Let $V^+, V^-$ (resp. $W^+, W^-$) be the $+1$- and $-1$-eigenspace of $\nu'$ (resp. $\nu$).
Then, the $+1$-eigenspace of $\tau$ is $V^+ \otimes W^+ +V^- \otimes W^-$,
and, after counting dimension, one finds that this eigenspace is more than half the dimension of $A$.
Consequently, $\tau$ is of orthogonal type.

In total, there is a surjective homomorphism $s: \Q[H] \to A \cong \Mat_{m}(B) \cong \Mat_{2m}(L)$ with
induced orthogonal involution on $\Mat_{2m}(L)$. Arguing as in the Section \ref{section:symgroup}, we
find that, up to finite index, $\Omega_{H,A} = \mc G^1(\frak O)$ is isomorphic to $\Sp(2 \cdot (2m) (g-1), \mc O) = \Sp(4m(g-1), \mc O)$.
Thus, there is a virtual epimorphism $\Mod(\Sigma_g)$ onto $\Sp(4m(g-1), \mc O)$.

For $m = 1$, the simple factor of $\Q[H]$ is $A = B$, and clearly $d(H) = 2$.
The analysis above applies
mutatis mutandis to establish that $\mc G^1(\frak O)$, up to finite index, is $\Sp(4(g-1), \mc O)$
and there is a virtual epimorphism $\Mod(\Sigma_g) \to \Sp(4(g-1), \mc O)$.

Now suppose $m = 2$. We construct $H$ as a subgroup of the units in $\Mat_2(B)$. Namely,
let $X = \tilde s(x) \in B$ and $Y = \tilde s(y) \in B$, and set
$$H = \left\{ \JMat{ X^{2k} Z & 0 \\ 0 & Z}, \JMat{ 0 & X^{2k+1} Z \\ Z & 0 } \in \Mat_2(B) \;\; | \;\; Z \in s(\Dih(2n)), k \in \Z \right\}  $$
This is a group which is finite and generated by the following two elements
$$ \JMat{ Y & 0 \\ 0 & Y} \;\;\;\;\; \JMat{ 0 & X \\ I & 0 }. $$
The inclusion $H \to \Mat_2(B)$ induces a homomorphism $\Q[H] \to \Mat_2(B) \overset{\text{def}}{=} A$
which is surjective. The induced involution $\tau$ satisfies $\tau( (b_{ij})) = (\nu(b_{ji}))$.
Using arguments similar to the above, we find that $\mc G^1(\frak O)$ is, up to finite index, $\Sp(8(g-1), \mc O)$
and there is a virtual epimorphism $\Mod(\Sigma_g) \to \Sp(8(g-1), \mc O)$.

\subsection{Example (c) of Theorem \ref{theorem:examplesofquotients}}
Here, we use the finite group $H = \Sym(m+1) \times \Cyc(n)$ for $m \geq 2$ and
$H = \Cyc(n)$ for $m = 1$ where $\Cyc(n) = \Z/n\Z$.
We first describe an irreducible representation of $\Cyc(n)$. Namely, there is an
irreducible representation $\Cyc(n) \to \Q(\zeta)^\times$ sending the generator
of $\Cyc(n)$ to $\zeta$, a primitive $n$th root of unity. The corresponding map
$\Q[\Cyc(n)] \to \Q(\zeta)$ is surjective. The induced involution, $\nu$
on $L = \Q(\zeta)$ is the unique one sending $\zeta$ to $\zeta^{-1}$
which is an involution of the second kind with fixed field $K = \Q(\zeta)^+$.

Suppose now that $m \geq 2$. The group $\Sym(m+1) \times \Cyc(n)$ has
generators $(\gamma, 0)$ and $(\delta, 1)$ where $\gamma$ is a transposition
and $\delta$ is an $m+1$-cycle. The conjugates of $(\gamma, 0)$ by $(\delta, 1)$
generate $\Sym(m+1) \times 0$, and from there it is obvious that $H$ is generated
by both elements.

Via an argument similar to the one used in Section \ref{section:exampleb}, we have a surjective homomorphism
$\Q[H] \to \Mat_m(\Q) \otimes \Q(\zeta) \cong \Mat_m(\Q(\zeta)) = A$. Since the center is $L = \Q(\zeta)$,
the induced involution on $A$ is of second kind by Proposition \ref{prop:typefromrealrep}. The
$\tau$-fixed subfield $K$ is $\Q(\zeta)^+$.

It only remains to describe $\Omega_{H,A}$. From the description of $\langle -,-\rangle$,
it is clear that a matrix $E$ is the adjoint of
$C \in \Mat_{2g-2}(A^{op})$ if and only if 
$$ E^* \JMat{ 0 & I \\ -I & 0}  = \JMat{ 0 & I \\ -I & 0} C$$
where $E^*$ is the map defined by $(e_{ij})^* = (\tau(e_{ji}))$.
Thus, the adjoint involution for $\langle -,-\rangle$ is 
$$\sigma(C) = \JMat{ 0 & I \\ -I & 0}^{-1} C^* \JMat{ 0 & I \\ -I & 0}.$$
Since $\Mat_{2g-2}(A^{op}) \cong \Mat_{2m(g-1)}(\Q(\zeta))$, the involution
$\sigma$ is also the adjoint involution for a nondegenerate $\Q(\zeta)$-valued Hermitian form
$\langle -,-\rangle_0$ on $\Q(\zeta)^{2m(g-1)}$ (\cite[Theorem 4.2]{Knusetal}). From the fact that 
$\langle -,-\rangle$ has a maximal isotropic submodule, one can deduce that $\langle -,-\rangle_0$ has a
maximal isotropic subspace (i.e. of $\Q(\zeta)$-dimension $m(g-1)$) as well. All Hermitian forms on
$\Q(\zeta)^{2m(g-1)}$ with a maximal isotropic subspace are equivalent (e.g. one can apply
the same arguments as in Lemma \ref{lemma:niceAibasis}), and so the form $\langle-,-\rangle_0$
is equivalent to the form defining $U(m(g-1), m(g-1), \Q(\zeta))$.

Thus, there is an isomorphism from a finite index subgroup of $\Omega_{H,A} = \mc G^1(\frak O)$
to a finite index subgroup of $\SU(m(g-1), m(g-1), \mc O)$ where $\mc O$ is the ring of integers in $\Q(\zeta)$.
By Theorem \ref{theorem:maintheoremprocedural}, there is a virtual epimorphism from $\Mod(\Sigma_g)$
onto $\SU(m(g-1), m(g-1), \mc O)$.

For the case $m=1$, a simplified version of the above argument applies, and so there is a virtual epimorphism
$\Mod(\Sigma_g)$ onto $\SU((g-1), (g-1), \mc O)$.

\subsection{Example (d) of Theorem \ref{theorem:examplesofquotients}}
For these examples, we use the 
finite group $H = \Alt(m+1) \times \Dic(4n)$ for $m \geq 3$ and $\Dic(4n)$ for $m = 1$ where $\Dic(4n)$ is the dicyclic group of order $4n$. 
The group $\Dic(4n)$ has the following presentation:
$$\Dic(4n) = \langle x, y \; | \; x^{2n} = 1, y^2 = x^n, y^{-1} x y = x ^{-1} \rangle. $$
For $m = 2$, we use a group containing $\Dic(4n)$ which we will describe below and which has a construction
similar to that for example (b).

We begin with the case $m \geq 3$. The group $H = \Alt(m+1) \times \Dic(4n)$ is generated
by the two elements $(\gamma, y)$ and $(\delta, x)$. This follows verbatim from the same argument 
for $\Alt(m+1) \times \Dih(2n)$.

The group ring $\Q[\Dic(4n)]$ has a representation onto a division algebra $D$ defined as follows.
Let $\zeta$ be the primitive $2n$th root of unity, and let $\overline{\phantom{\alpha}}$ denote the unique order $2$ automorphism
of $\Q(\zeta)$ sending $\zeta \mapsto \zeta^{-1}$. The division algebra is 
$$D = \left\{ \left( \begin{array}{rr} \alpha & \beta \\ - \overline{\beta} & \overline{\alpha} \end{array} \right) \in \Mat_2(\Q(\zeta))
 	\; | \; \alpha, \beta \in \Q(\zeta) \right\}.$$
The center $L$ of $D$ is $\Q(\zeta)^+$, and the epimorphism $\tilde s: \Q[\Dic(4n)] \to D$ is defined on the generators by 
$$\tilde s(x) =  \left( \begin{array}{rr} \zeta & 0 \\ 0 & \overline{\zeta} \end{array} \right) \;\;\;\;
      \tilde s(y) =  \left( \begin{array}{rr} 0 & 1 \\ -1 & 0 \end{array} \right)$$
The induced involution, call it $\nu$, on $D$ is that involution satisfying $\nu(\tilde s(z))  = \tilde s(z)^{-1}$ 
for all $z \in \Dic(4n)$. Thus, in terms of the matrices, $\nu$ is conjugate transpose, and it is of first kind (on $D$). The fixed subspace
of $\nu$ has $L$-dimension $1$ and $D$ has $L$-dimension $4$, so $\nu$ is of symplectic type.

We describe the simple factor $A$ of $\Q[H]$. Let $\tilde s': \Q[\Alt(m+1)] \to \Mat_m(\Q)$ be
the homomorphism from Section \ref{section:symgroup}, and let $\nu'$ be the involution of 
$\Mat_m(\Q)$ induced by the standard involution of $\Q[\Alt(m+1)]$.
Let $A = \Mat_m(\Q) \otimes_\Q D \cong \Mat_m(D)$.
There is a surjective homomorphism 
$\tilde r: \Q[H] \to A$ defined by $\tilde r(z, w) = \tilde s'(z) \otimes \tilde s(w)$ on $(z,w) \in H$.
The induced involution, $\tau$, on $A$ is equal to $\nu' \otimes \nu$. By an argument similar
to the one in Section \ref{section:exampleb}, one can show $\tau$ is of first kind and
symplectic type.
By Theorem \ref{theorem:Gtype}, $\mc G = \Aut_A(A^{2g-2}, \langle-,-\rangle)$
is the $L$-points ($K = L$ here) of an $L$-defined algebraic group whose real form
is $\SO(2m(g-1), 2m(g-1), \R)$. Thus, there is a virtual epimorphism of
$\Mod(\Sigma_g)$ onto $\Omega_{H,A} = \mc G^1(\frak O)$, an arithmetic subgroup of $\SO(2m(g-1), 2m(g-1), \R)$.

For $m = 1$, it is clear that $H = \Dih(2n)$ is generated by two elements, and the above
arguments apply mutatis mutandis to show that there is a virtual epimorphism of 
$\Mod(\Sigma_g)$ onto an arithmetic subgroup of $\SO(2(g-1), 2(g-1),  \R)$.

For $m = 2$, we construct $H$ as a subgroup of the units of $\Mat_2(D)$. Namely, 
let $X = s(x) \in D$ and $Y = s(y) \in D$, and set
$$H =  \left\{ \JMat{ X^{2k} Z & 0 \\ 0 & Z}, \JMat{ 0 & X^{2k+1} Z \\ Z & 0 }  \in \Mat_2(D)\;\; | \;\; Z \in s(\Dic(4n)), k \in \Z \right\} . $$
Like the group $H$ for $m=2$ in Section \ref{section:exampleb}, $H$ here is generated by two elements,
and the inclusion $H \to \Mat_2(D)$ induces a homomorphism $\Q[H] \to \Mat_2(D) \overset{\text{def}}{=} A$
which is surjective.  The induced involution $\tau$ satisfies $\tau( (b_{ij})) = (\nu(b_{ji}))$, and so $\tau$
is of symplectic type. Using arguments similar to the above, we find a virtual epimorphism onto an arithmetic
subgroup of $\SO(4(g-1), 4(g-1),  \R)$.

\subsection{Proof of Corollary \ref{corollary:surjectfinitegroups}} \label{section:ontoallfingrp}
This follows rather quickly from example (a) of Theorem \ref{theorem:examplesofquotients}.
A finite index subgroup of $\Sp(2m(g-1), \Z)$ is mapped onto $\Sp(2m(g-1), \Z/p\Z)$
for almost all primes $p$. The latter contains $\SL(m(g-1), \Z/p\Z)$ and hence contains
$\Sym(m(g-1) - 1)$. Now, every finite group $G$ embeds in $\Sym(m(g-1) - 1)$ for 
sufficiently large $m$. Thus, $\Mod(\Sigma_g)$ has a finite index subgroup
which is mapped onto $G$.

\subsection{Simple components of $\Q[H]$ in general}
From Theorem \ref{theorem:maintheoremprocedural}, given a finite group $H$ 
with $d(H) < g$ and
a simple factor $A$ of $\Q[H]$, there is a corresponding arithmetic group
$\Omega_{H, A}$ which is a virtual quotient of $\Mod(\Sigma_g)$.
It is therefore of interest to understand what kinds of algebras with involution
$(A, \tau)$ are obtained in this way. This seems to be an open problem in general.

We start with some basic properties of the pairs $(A, \tau)$. 
Recall that $A$ is isomorphic to $\Mat_m(D)$ for some integer $m$ and
some division algebra $D$ with a number field $L$ as center. Recall from
Lemma \ref{lemma:cyccenter}
that $L$ must be a subfield of a cyclotomic field. This is equivalent to $L$ being an
abelian extension of $\Q$.
Recall, moreover, from Proposition \ref{prop:typefromrealrep} that $L$ is totally
real if and only if $\tau$ is of first kind and $L$ is totally imaginary if and only if
$\tau$ is of second kind.

We will not say much about the involution $\tau$. However, in the case where $\tau$ is of first kind,
the division algebra $D$ is considerably restricted. The following proposition is essentially the
Brauer-Speiser Theorem, and the proof we give is basically that given by Fields in \cite{Fiel}.
Recall that the degree $d$ of a division algebra $D$ is the integer satisfying $d^2 = \dim_L(D)$
where $L$ is the center of $D$.

\begin{prop}
If $A = \Mat_m(D)$ has an involution of the first kind, then the degree of $D$ is at most $2$, i.e.
$D$ is the field $L$ or a quaternion algebra over $L$.
\end{prop}
\begin{proof}
If $\tau$ is an involution of the first kind, then $\tau$ is an isomorphism $A \cong A^{op}$
as $L$-algebras. Thus $A \otimes_L A \cong A \otimes_L A^{op} \cong \Mat_{n}(L)$
where $n = \dim_L(A)$ \cite[Proposition 3.12]{FarbDenn}. Thus, the class defined by $A$
in the Brauer group of $L$ has order $1$ or $2$. (See below for the definition of the Brauer group
and \cite{FarbDenn} for more details.) For number fields $L$, the exponent of $\Mat_m(D)$
in the Brauer group of $L$ is equal to the degree of $D$ \cite[Theorem 32.19]{Rein}.
\end{proof}

Let us consider, now, the simpler question of which algebras $A$ appear as simple
components of $\Q[H]$ or what division algebras $D$ appear in such an $A \cong \Mat_m(D)$
regardless of $m$.
Notice that, by using $\Sym(\ell+1)$ and its standard representation, given any algebra
$A = \Mat_m(D)$ which is a simple part of $\Q[H]$, we can also obtain $\Mat_{\ell m}(D)$
as a simple part of $\Q[H \times \Sym(\ell+1)]$. One can show that this increases the requisite number of
generators of the finite group by at most $1$ except when $H$ is trivial. For $D=L$ a field, the
algebra $\Mat_m(D)$ is a simple component of $\Q[H]$ for some $H$ precisely when $L$
is a subfield of a cyclotomic field (See e.g. \cite[Theorem 3.2]{Moll}). This leads to the proof of Theorem \ref{theorem:examplesofquotients2}.

\begin{proof}[Proof of Theorem \ref{theorem:examplesofquotients2}]
We use Theorem \ref{theorem:maintheoremprocedural} and follow a similar program as for Theorem \ref{theorem:examplesofquotients}.
From the above discussion, it follows that for some $s_L \in \N$, there is a finite group $H$ such that 
$\Mat_{s_L}(L)$ is a simple component of $\Q[H]$. Consequently, using arguments similar to those for
Theorem \ref{theorem:examplesofquotients}, $\Mat_{m s_L}(L)$ is a simple part of $\Q[H \times \Sym(m+1)]$.
Let $N = d(H \times \Sym(m+1)) + 1$.

Suppose $L$ is a totally real field. By Proposition \ref{prop:typefromrealrep}, $\tau$ is of first kind and $K = L$.
By the same proposition, since $\Mat_{m s_L}(L) \otimes_K \R = \Mat_{m s_L}(K) \otimes_K \R =\Mat_{m s_L}(\R)$, the involution $\tau$ is
of orthogonal type. Arguing as for example (b) of Theorem \ref{theorem:examplesofquotients}, we find a virtual
epimorphism of $\Mod(\Sigma_g)$ onto $\Sp(2 m s_L (g-1), \mc O)$ where $\mc O$ is the ring of integers
in $L$.

Suppose $L$ is a totally imaginary field. By Proposition \ref{prop:typefromrealrep}, $\tau$ is of second kind.
Arguing as for example (c) of Theorem \ref{theorem:examplesofquotients}, we find a virtual
epimorphism of $\Mod(\Sigma_g)$ onto $\SU(m s_L (g-1), m s_L (g-1), \mc O)$ where $\mc O$ is the ring of integers
in $L$.
\end{proof}

If we ask only what division algebras $D$ appear regardless of $m$, then we can appeal
to results on the Schur subgroup, $S(L)$, of the Brauer group, $B(L)$, for a finite abelian extension
$L$ of $\Q$. The elements of the Brauer group are equivalence classes of central simple
$L$-algebras where $A_1$ and $A_2$ are equivalent if they are isomorphic,
respectively, to $\Mat_{m_1}(D)$ and $\Mat_{m_2}(D)$ for some division algebra $D$.
The group operation is $\otimes_L$. The Schur group consists of those classes
$[A]$ where $A$ is a simple component of some $L[H]$.
In the case where $L$ is a finite abelian extension of $\Q$, the Schur subgroup is also
the set of those classes $[A]$ where $A$ is a simple component of some $\Q[H]$.
(This seems to be well-known among experts, but is also a consequence of the stronger
\cite[Theorem 3.2]{Moll}.)

We mention some of the easier-to-state theorems in the literature.
One basic question that can be answered is what possible degree $d$ of a division algebra $D$ is possible.
I.e. for what $d$ is $\dim_L(D) = d^2$ for some division algebra $D$ in the Schur group of $L$? It follows
from \cite{BenaScha} that, for $L$ containing a primitive $d$th root of unity,
there are infinitely many elements $[D]$ in $S(L)$ where $D$ has degree $d$. In particular, in
combination with our theorems, this implies that for each of these infinitely many
division algebras $D$, some $m$ and sufficiently large $g$ (depending on $D$), there is a
virtual epimorphism of $\Mod(\Sigma_g)$ onto the arithmetic group
$\mc G^1(\frak O)$, i.e. the elements of reduced norm $!$ in $\Aut_{\frak O}(\frak O^{2g-2}, \langle-,-\rangle)$ 
where $\frak O$ is an order in $\Mat_m(D)$.

In the opposite direction,
a result due to Schacher and Fein (independently) tells us that for infinitely many 
division algebras $D$ are unattainable regardless of $m$ \cite{Scha}. 
In the special case of $L = \Q$, it has been shown that $S(\Q)$ and $B(\Q)$ are equal when restricting
to classes $[D]$ where the degree of $D$ is $2$ \cite{Bena, Fiel}. Some other papers
characterize $S(L)$ for certain subclasses of fields in cyclotomic fields (e.g.\cite{BenaScha, Janu}).
We refrain from going into further detail and refer the interested reader to the literature.

\bibliographystyle{plainnat}
\bibliography{biblio}
\end{document}